\documentclass{article}
\usepackage{amsthm, amssymb, amsmath}
\usepackage{enumerate, mathtools}
\usepackage{graphicx}
\usepackage{subcaption}
\usepackage{xcolor}
\numberwithin{equation}{section}
\newtheorem{theorem}{Theorem}[section]
\newtheorem{Assumption}[theorem]{Assumption}
\newtheorem{lemma}[theorem]{Lemma}

\newtheorem{proposition}[theorem]{Proposition}

\newtheorem{definition}[theorem]{Definition}
\huge

\date{}

\begin{document}
\title{On the continuity of phase transition of three-dimensional square-lattice XY models}
\author{Zhou Gang}
\maketitle
\centerline{Department of Mathematics and Statistics, Binghamton University, Binghamton, NY, 13850}
\setlength{\leftmargin}{.1in}
\setlength{\rightmargin}{.1in}
\normalsize \vskip.1in
\setcounter{page}{1} \setlength{\leftmargin}{.1in}
\setlength{\rightmargin}{.1in}
\large

\date

\setlength{\leftmargin}{.1in}
\setlength{\rightmargin}{.1in}
\normalsize \vskip.1in
\setcounter{page}{1} \setlength{\leftmargin}{.1in}
\setlength{\rightmargin}{.1in}
\large

\section*{Abstract}
We study the continuity of magnetization at the phase transition of the ferromagnetic XY model in the three-dimensional square lattice with the nearest neighborhood interaction. We assume that, at the critical temperature, with probability 1, for every edge in the infinite directed graph generated by the random path representation, finitely many edges exist so that they form a finite loop.
Then, we prove that the phase transition is continuous at the critical temperature. The main technical contribution is to find a switching lemma to establish a bijection between equally weighted graphs.

\tableofcontents

\section{Introduction}
In the present paper, we are interested in the phase transition at the critical inverse temperature of the XY model on the three-dimensional square lattice.

To formulate the problem we start by considering a finite subset of a three-dimensional square lattice 
\begin{align}
\mathcal{L}=[-L,L]^3\bigcap \mathbb{Z}^3.
\end{align} On this lattice the Hamiltonian is defined as
\begin{align}\label{eq:Hamil}
H_{L,\nu}:=-\sum_{k,l\in \mathcal{L}} J_{k,l}S_k \cdot S_l
\end{align} and $\nu=+$ or $0.$
Here, depending on the location of the site $k$, $S_k\in \mathbb{S}^1$ satisfies different conditions: the lattice is decomposed into two parts,
\begin{align}
    \mathcal{L}=\mathcal{L}^{o}\cup \partial \mathcal{L},
\end{align} with $\mathcal{L}^{o}$ being the interior, and $\partial \mathcal{L}=\{(z_1,z_2,z_3)\ | \ \max\{|z_1|,\ |z_2|,\ |z_3|\}=L\}$ the boundary;
for $k\in \mathcal{L}^{o} $, the only requirement for $S_{k}$ is that $|S_k|=1;$ when $k\in \partial \mathcal{L}$ i.e. it is on the boundary, we consider two boundary conditions: when $\nu=0$, we use the free boundary condition, specifically
\begin{align}
S_k=0, \ \text{if}\ k \in \partial\mathcal{L};
\end{align}
when $\nu=+$, the plus boundary condition,
\begin{align}
S_{k}=(1,0)^{T}, \ \text{if}\ k \in \partial\mathcal{L}.
\end{align}
Here we only consider the nearest neighborhood interaction, specifically,
\begin{align}
J_{k,l}=\left\{
\begin{array}{cc}
\frac{1}{6}\ & \text{if}\ |k-l|=1;\\
0\ & \text{otherwise}.
\end{array}
\right.
\end{align} Thus for any fixed $k$,
$
\sum_{l}J_{k,l}=1.
$ Also, we remark that all the analyses in the present papers work for finite-ranged, symmetric ferromagnetic potentials.

Next, we define a probabilistic measure.

For any given local function $F$ of the spins, we define
\begin{align}\label{def:Finite}
\langle F\rangle_{\beta,\nu, L}:=\int \prod_{k\in \mathcal{L}^o} dA_{k} \ e^{- \beta H_{L,\nu}(S)}F(S)
\end{align} where $A_k$ is the area element for $S_k$; $\nu=+\ \text{or}\ 0;$ and $H_{L,\nu}$ denotes the Hamiltonian in \eqref{eq:Hamil}.
By choosing $F=1,$ we define
\begin{align}\label{def:zbkl}
Z_{\beta,\nu,L}:=\langle 1\rangle_{\beta,\nu,L}.
\end{align}

We are interested in the following quantities:
\begin{align}\label{def:expectation}
\big\langle F\big\rangle_{\beta,\nu}:=\lim_{L\rightarrow +\infty} \frac{\big\langle F\big\rangle_{\beta,\nu, L}}{Z_{\beta,\nu, L}}.
\end{align} And we are especially interested in the case $F=\cos(\theta_0)$. The so-called magnetization $m^*: [0,\infty)\rightarrow [0,\infty)$ is a function defined as
\begin{align}\label{def:mStBeta}
    m^*(\beta):=\Big\langle \cos(\theta_0)\Big\rangle_{\beta,+}.
\end{align}
It is well known that there exists a critical inverse temperature $\beta_c>0$ such that 
when $\beta<\beta_c$, $m^*(\beta)\equiv 0;$
and when $\beta>\beta_c$, $m^*(\beta)>0.$ 

A natural question is:  \begin{align}\label{eq:proIn}
\text{at the critical inverse temperature}\ \beta=\beta_c, \ \ m^*(\beta_c)=0?
\end{align} A positive answer implies the continuity of $m^{*}(\beta)$ at $\beta=\beta_c$ because (1) $\lim_{\beta \searrow \beta_c} m^{*}(\beta)=m^{*}(\beta_c)$; and (2) $m^{*}(\beta)=0$ if $\beta<\beta_c$.

We will consider the problem \eqref{eq:proIn} in the present paper.
Here we restrict our consideration to the 3-dimensional square lattice and nearest neighbor interaction. We need to assume that Assumption \ref{assu:uniqClus} holds, which states that, at the critical inverse temperature, in the directed graphs generated by the random path representation, with probability 1, for every fixed edge, one can find finitely many edges so that together they form a finite loop.
Under these conditions, we will prove that
\begin{align}
m^*(\beta_c)=0.
\end{align}

We briefly review some relevant results.
For the two-dimensional cases, there is the Mermin–Wagner theorem \cite{MerminWagner1966}.
For the long-ranged potential Ising model, it is known that magnetization at the critical inverse temperature might not be continuous; see, e.g. \cite{Aizenman1988DiscontinuityOT}.

Now we discuss some technicality. Perhaps the Ising model is the most well-studied classical spin model. A key tool is the so-called switching lemma, which Griffiths, Hurst, and Sherman invented \cite{GHS70} in 1970, and was used by Aizenman in \cite{AIZ82}, and then it was used extensively, see e.g. \cite{AizenmanFernandez1986, ADS2015, AizenDum2021,MR4072233}. For example, this technique has been successfully applied elsewhere e.g. \cite{SPST2021,werner2025switching}.

Given the significance of the switching lemma, it is important to find a similar technique for the other spin models, such as the XY model, which is to be considered in this paper. The present paper is centered around building a ``switching lemma" for directed graphs generated by the random path representation of the XY model.

Before discussing the difficulties, we briefly review some key steps in applying the switching lemma for undirected graphs. 

We consider two sets of pairs of (undirected) graphs on a lattice $\Big(\mathbb{Z}\cap [-N,N]\big)^{3}$ for some large integer $N$,  and a possible so-called ghost $\delta-$site, see e.g. the graphs generated by the random current representation of Ising model in \cite{ADS2015}. Consider the following two sets $\Lambda_{x,y}$ and $\Gamma_{x,y}$ defined as
\begin{align}\label{eq:undirectedDef}
\begin{split}
    \Lambda_{x,y}&:=\Big\{\big(A_{\delta},\ B\big)\ \Big|\ A_{\delta}\ \text{and}\ B \ \text{satisfy the conditions in (1), see below}\Big\};\\
    \Gamma_{x,y}&:=\Big\{\big(C_{\delta}, \ D\big)\ \Big|\ C_{\delta}\ \text{and}\ D \ \text{satisfy the conditions in (2), see below}\Big\},
\end{split}
\end{align} where $x$ and $y$ are two sites on the lattice, and the two conditions are
\begin{itemize}
    \item[(1)] $A_{\delta}$ and $B$ only have finitely many edges; $A_{\delta}=A_{\delta}(x,y)$ has sinks at $x$ and $y$, and $B$ has no sink, $A_{\delta}$ might depend on the $\delta$-site. And their union $A_{\delta}\cup B$ contains a $x\longleftrightarrow y$ path, and this path avoids the $\delta$-site. ----- In the rest of the paper we say this path is $\delta$-site avoiding.
    \item[(2)] $C_{\delta}$ and $D$ only have finitely many edges; $C_{\delta}$ has no sink; $D=D(x,y)$ has sinks at $x$ and $y$, $C_{\delta}$ might depend on the $\delta$-site. Their union $C_{\delta}\cup D$ contains a $\delta-$site-avoiding $x\longleftrightarrow y$ path.
\end{itemize}

The result is the following:
\begin{lemma}\label{LM:UndirSwitch}
Under the conditions above, there exists a bijection between $\Lambda_{x,y}$ and $\Gamma_{x,y}$.
\end{lemma}

The proof is based on the so-called switching lemma, and we will sketch it. 

Based on the information above, there exist countably many $x\longleftrightarrow y$-paths, denoted by $P_1$, $P_2$, $\cdots$, such that (1) they avoid the $\delta-$site, and (2) $\Lambda_{x,y}$ and $\Gamma_{x,y}$ can be partitioned into the following disjoint subsets:
\begin{align}
\begin{split}\label{eq:partition}
    \Lambda_{x,y}=&\cup_{k=1}^{\infty}\Lambda_{x,y;P_k},\\
    \Gamma_{x,y}=&\cup_{k=1}^{\infty}\Gamma_{x,y;P_k}
\end{split}
\end{align} where $\Lambda_{x,y;P_k}$ and $\Gamma_{x,y;P_k}$ are sets of graphs defined as
\begin{align}
\begin{split}\label{eq:twoSets}
    \Lambda_{x,y;P_k}:=&\Big\{\big(A_{\delta},\ B\big)\in \Lambda_{x,y} \ \Big|\ P_k\subset A_{\delta}\cup B;\ \text{and for any}\ l<k,\ P_l\not\subset A_{\delta} \cup B\Big\};\\
    \Gamma_{x,y;P_k}:=&\Big\{ \big(C_{\delta}, \ D\big) \in \Gamma_{x,y} \ \Big| \ P_k\subset C_{\delta} \cup D;\ \text{and for any}\ l<k,\ P_l\not\subset C_{\delta}\cup D\Big\}.
\end{split}
\end{align}
We can define a bijection $$F_{P_k}:\ \Lambda_{x,y;P_k}\rightarrow\Gamma_{x,y;P_k},$$ as follows: 
\begin{align}\label{eq:piecewise}
F_{P_{k}}(A_{\delta},B)=(C_{\delta},D)
\end{align} where $C_{\delta}$ and $D$ are defined as
\begin{align*}
\begin{split}
    C_{\delta}:=&(A_{\delta}\backslash P_{k})\cup (B\cap P_k),\\
    D:=&(B\backslash P_{k})\cup (A_{\delta} \cap P_k).
\end{split}
\end{align*}
It is easy to see that the map is a bijection. A very useful property is that there is a ``conservation law":
\begin{align}\label{eq:conserv}
    A_{\delta}\cup B=C_{\delta}\cup D, 
\end{align} which, together with \eqref{eq:twoSets}, guarantees that
\begin{align}\label{eq:interseEmp}
    F_{P_k} (\Lambda_{x,y;P_k})\cap F_{P_l} (\Lambda_{x,y;P_l})=\emptyset, \text{if}\ k\not=l.
\end{align}

\eqref{eq:partition}-\eqref{eq:interseEmp} establish a bijection between $\Lambda_{x,y}$ and $\Gamma_{x,y}$. Therefore, Lemma \ref{LM:UndirSwitch} holds.

Next, we consider directed graphs generated by the random path representation of the XY model.

In brief, for directed graphs, a simple generalization of the technique above does not work, especially because we do not have a ``conservation law" like \eqref{eq:conserv}. 

Specifically, if $(A_{\delta},B)$ contains $x\rightarrow y$ path $P$ and it avoids the $\delta$
site, then the switching  $F_P$ works as follows:
\begin{align}\label{eq:switchDir}
    F_p\Big((A_{\delta},B)\Big)=\Big(C_{\delta},D\Big)
\end{align}
where $C_{\delta}$ and $D$ are defined as
\begin{align*}
    C_{\delta}:=&(A_{\delta}\backslash P)\cup (\overline{B\cap P}),\\
    D:=&(B\backslash P)\cup (\overline{A_{\delta}\cap P}).
\end{align*}
Here $\overline{E}$ is the reverse of the directed graph $E.$ Consequently, instead of the conservation law \eqref{eq:conserv}, in general,
\begin{align}
    A_{\delta}\cup B \not= C_{\delta}\cup D.
\end{align} Consequently, a simple generalization of switching lemma will not work for the XY model. 
For the details, see \eqref{eq:adverse1} and \eqref{eq:adverse2} below. 

This new difficulty forces us to adopt the technique of edge-pairing used in \cite{chayes2000intersecting, van2021elementary,benassi2020loop}. The idea is the following: provided that there exists a paired $x\rightarrow y$ path $P$ and it avoids the $\delta-$site, we can switch this path to establish some bijection between certain sets, see \eqref{eq:pairedPSwi} below.

However, the edge-pairing has its limitations and needs to be refined to solve the present problem. Specifically, we can not guarantee that a useful paired path exists, for example, if all the $x\rightarrow y$ path passes through the $\delta$ site, then we have no path to switch to begin with. To overcome this difficulty, we need the following:
\begin{itemize}
    \item[(1)]
We need Assumption \ref{assu:uniqClus} to ensure that, as $L\rightarrow \infty,$ every graph has a finite $x\rightarrow y$ path, almost surely, so that we can at least have a path for switching, see Proposition \ref{prop:uniq} and Lemma \ref{LM:GraWithSink} below.
\item[(2)]
We pair the edges indiscriminately and then refine the pairing to ensure that, almost surely, every graph has a switchable paired $x\rightarrow y$ path, see \eqref{eq:finitelyLooped} below.
\item[(3)] After refining the pairing, we prove that the switching is between graphs of equal weights. For this we use the notion of ``surgical switching", for the detail, we refer to Proposition \ref{prop:keyWeights} below.
\end{itemize}

The paper is organized as follows. The main results will be stated in Section \ref{sec:MainRes}. The directed graphs generated by the random path representation will be derived in Section \ref{sec:directedGraph}. One of the main results, Theorem \ref{MainTHM:first}, will be reformulated and proved in Section \ref{sec:firstTHM} by assuming a technical result Proposition \ref{Prop:limitDxy}. Another main result, Theorem \ref{THM:second}, will be proved in Section \ref{sec:secondTHM}. In the rest of the paper, we prove Proposition \ref{Prop:limitDxy}.
\section*{Acknowledgement}
The author wishes to thank J\"urg Fr\"ohlich for suggesting the problem and for many stimulating discussions. The author is partly supported by Simons grant \#709542.

\section{Main Results}\label{sec:MainRes}
We start with defining a function
\begin{align}
\widetilde{M}_n(\beta):= \frac{1}{|B_n|^2} \sum_{x,y\in B_n}\langle S_x\cdot S_y\rangle_{0,\beta},
\end{align} where $B_n$ is a three-dimensional finite square lattice
\begin{align*}
    B_{n}:=\Big\{ (z_1,z_2,z_3)\in \mathbb{Z}^{3}\ \Big| \ |z_k|\leq n, k=1,2,3\Big\},
\end{align*} and recall that $\langle f\rangle_{0,\beta}$ is the expectation for free boundary condition at the inverse temperature $\beta.$

The first result is:
\begin{theorem}\label{THM:second}
At the critical inverse temperature $\beta=\beta_c$, 
\begin{align}
    \lim_{n\rightarrow\infty} \widetilde{M}_n(\beta_c)=0.
\end{align}
\end{theorem}
The theorem will be proved in Section \ref{sec:secondTHM}, based on the results in \cite{FSS76,FILS78}, see also \cite{Biskup09, FV2018}.

The next result, Theorem \ref{MainTHM:first}, will be built on the following assumption.
\begin{Assumption}\label{assu:uniqClus}
At the critical inverse temperature $\beta=\beta_c$, for any directed graph $G$ generated by the random path representation, of free- and plus- boundary conditions on the infinite three-dimensional square lattice, every fixed edge $e_{a\rightarrow b}\in G$ is part of a finite loop, with probability 1.

\end{Assumption}

To illustrate the ideas we use the following two examples. The first graph is a good example because every edge belongs to a finite loop. The second one might not be good because, if the $a\rightarrow \infty$ and $\infty\rightarrow b$ paths don't cross, then the edge $b\rightarrow a$ does not belong to a finite loop.
\begin{figure}[ht]
  \subcaptionbox*{A good example}[.55\linewidth]{%
    \includegraphics[width=\linewidth]{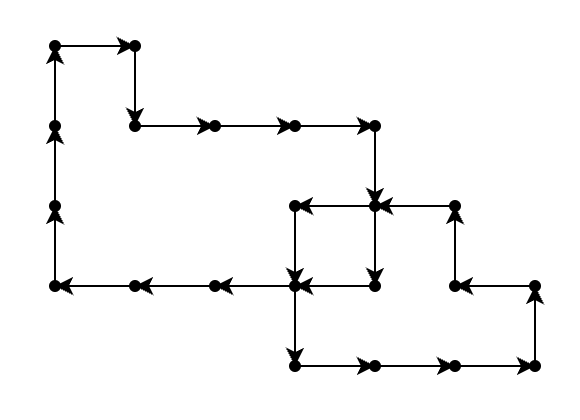}%
  }%
  \hfill
  \subcaptionbox*{A bad example}[.40\linewidth]{%
    \includegraphics[width=\linewidth]{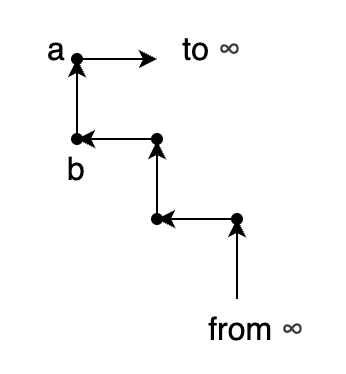}%
  }
  \caption{Two images}
\end{figure}

We continue to formulate our problem.

Assumption \ref{assu:uniqClus} is important because it implies Proposition \ref{prop:uniq} and Lemma \ref{LM:GraWithSink}, which in turn implies that there is at least a useful path for us to switch. The graphs will be derived in Section \ref{sec:directedGraph}.

It seems that Assumption \ref{assu:uniqClus} is reasonable. 
In studying the percolation problem for directed graphs, Grimmet and Hiemer proved, in \cite{Grimmett02}, that: for subcritical and critical cases, with probability zero, there exists an infinite open path.

Provided that we could prove the same results for our cases, specifically, with probability zero, a fixed $e_{a\rightarrow b}$ edge is part of the path reaching $\infty,$ then it implies that Assumption \ref{assu:uniqClus} holds at $\beta\leq \beta_c$: when there is no infinite path almost surely, the graphs are localized. Since our graphs are without sinks or sources, the assumption holds.

Unfortunately, for technical reasons, we can not prove the Assumption at the moment, for example, the FKG inequality used in \cite{Grimmett02} is not available here.

We are ready to state the main theorem of the present paper. Recall the definition of magnetization $m^*(\beta)$ from \eqref{def:mStBeta}.
\begin{theorem}\label{MainTHM:first}
Under Assumption \ref{assu:uniqClus}, at the inverse temperature $\beta=\beta_c$,
\begin{align}\label{eq:mbeta}
m^*(\beta_c)=0.
\end{align}
\end{theorem}
This theorem will be proved in Section \ref{sec:firstTHM}.

In the proof of Theorem \ref{MainTHM:first}, we will use some of the ideas in \cite{ADS2015}.

\section{Directed Graphs Generated By Random Path Representation}\label{sec:directedGraph}
Here we start with rewriting \eqref{def:Finite} by the so-called random path representation. Among the known representations, see \cite{FroSpe81, FroSpe83, benassi2017correlation}, we choose the random path representation because it is easier to establish a ``switching" lemma.

\subsection{Free boundary condition}
We rewrite Hamiltonian as
\begin{align}\label{eq:Hamil2}
H_{L,0}=-\sum_{k,l\in \mathcal{L}} J_{k,l}S_{k}\cdot S_{l}
=-\frac{1}{2}\sum_{k,l\in \mathcal{L}} J_{k,l}\Big( e^{i(\theta_k-\theta_l)}+e^{-i(\theta_k-\theta_l)}\Big),
\end{align}
where, in the second step we used $S_k=(\cos\theta_{k},\ \sin\theta_k)$ and
\begin{align}
\begin{split}
S_k\cdot S_l=&\cos\theta_k \cos\theta_l+\sin\theta_k\sin\theta_l\\
=&\cos(\theta_k-\theta_l)\\
=&\frac{1}{2}\big(e^{i(\theta_k-\theta_l)}+e^{-i(\theta_k-\theta_l)}\big).
\end{split}
\end{align}
Decompose the Hamiltonian $H$ into two parts
\begin{align}
H_{L,0}=H_{+}+H_{-},
\end{align} where $H_{+}$ and $H_{-}$ are defined as
\begin{align}
\begin{split}
H_{+}:=&-\sum_{\sigma(k)>\sigma(l) } J_{k,l} e^{i(\theta_k-\theta_l)},\\ H_{-}:=&-\sum_{\sigma(k)>\sigma(l) } J_{k,l} e^{-i(\theta_k-\theta_l)},
\end{split}
\end{align} here $\sigma: \mathcal{L}\rightarrow \mathbb{Z}$ is an (arbitrary) injective map.
Recall that $J_{k,k}=0$ and $J_{k,l}=J_{l,k}.$

Taylor-expand the exponentials $e^{-\beta H_{+}}$ and $e^{-\beta H_{-}}$ to obtain 
\begin{align}
\begin{split}
    e^{-\beta H_{+}}=&\sum_{n_{k\rightarrow l}}\prod_{\sigma(k)>\sigma(l) }(\beta J_{k,l})^{n_{k\rightarrow l}}\frac{1}{n_{k\rightarrow l}!}e^{i n_{k\rightarrow l}(\theta_k-\theta_l)};\\
    e^{-\beta H_{-}}=&\sum_{ n_{k\leftarrow l}}\prod_{\sigma(k)>\sigma(l) }(\beta J_{k,l})^{n_{k\leftarrow l}}\frac{1}{n_{k\leftarrow l}!}e^{-in_{k\leftarrow l}(\theta_k-\theta_l)}.
\end{split}
\end{align}
Consequently,
\begin{align}
\begin{split}
e^{-\beta H_{L,0}}=&e^{-\beta H_{+}} e^{-\beta H_{-}}\\
=&\sum_{n_{k\rightarrow l},\ n_{k\leftarrow l}}\prod_{\sigma(k)>\sigma(l) }(\beta J_{k,l})^{n_{k\rightarrow l}+n_{k\leftarrow l}}\frac{1}{n_{k\rightarrow l}!}\frac{1}{n_{k\leftarrow l}!}e^{i(n_{k\rightarrow l}-n_{k\leftarrow l})(\theta_k-\theta_l)}.
\end{split}
\end{align} We call the nonnegative integer $n_{k\rightarrow l}$ the number of edges going from $k$ to $l$; and $n_{k\leftarrow l}$ the number of edges going from $l$ to $k.$ 

Take the integral $d^3 \theta_k$ for any $k\in \mathbb{Z}^3$ and find that,
\begin{align}
 \int_{0}^{2\pi} e^{-\beta H_{L,0}} d^3 \theta_k\not=0
\end{align} only when the numbers of incoming edges and outgoing ones are equal, equivalently:
\begin{align}
 \sum_{j:\ \sigma(k)>\sigma(j) }n_{k\rightarrow j}+\sum_{l:\ \sigma(k)>\sigma(l)}n_{l\leftarrow k}=\sum_{l:\ \sigma(l)>\sigma(k)}n_{l\leftarrow k}+\sum_{j:\ \sigma(j)>\sigma(k) }n_{j\rightarrow k}.\label{eq:sumR}
\end{align}
Therefore, the quantity $Z_{\beta,0,L}$, defined in \eqref{def:zbkl}, takes a new form
\begin{align}
\begin{split}\label{eq:comp}
Z_{\beta,0,L}=&\frac{1}{(2\pi)^{\#\Lambda}} \int e^{-\beta H_{L,0}} \prod_{k\in L^0}d^3 \theta_{k}\\
=&\left\{
\begin{array}{ll}
\displaystyle\sum_{n_{k\rightarrow l},\ n_{k\leftarrow l}}\prod_{k,l}\Big[(\beta J_{k,l})^{n_{k\rightarrow l}+n_{k\leftarrow l}}\frac{1}{n_{k\rightarrow l}!}\frac{1}{n_{k\leftarrow l}!}\Big]\ \text{if}\ \eqref{eq:sumR}\ \text{holds for all}\ k,\\
\\
0\ \text{otherwise}.
\end{array}
\right.
\end{split}
\end{align}

To simplify the notations, we use the following expression: provided that \eqref{eq:sumR} holds,
\begin{align}\label{eq:fdeltal01}
\begin{split}
    Z_{\beta,0,L}=&\sum_{ \partial({\bf{n}_{0,\rightarrow}+\bf{n}_{0,\leftarrow}})=\varnothing}\big(\beta\bf{J}\big)^{{\bf{n}}_{0,\rightarrow}+{\bf{n}}_{0,\leftarrow}}
\frac{1}{{\bf{n}_{0,\rightarrow}}!} \frac{1}{{\bf{n}_{0,\leftarrow}}!}\\
&=\sum_{ \partial{\bf{n}_{0}}=\varnothing}\big(\beta\bf{J}\big)^{{\bf{n}}_{0}}
\frac{1}{{\bf{n}_{0}!}} 
\end{split}
\end{align}
where $\bf{J}$, $\bf{n}_{0,\rightarrow}$ and $\bf{n}_{0,\leftarrow}$ are vectors defined as
\begin{align}\label{def:vectors1}
    {\bf{J}}:=\Big(J_{k,l}\Big)_{(k,l):\ \sigma(k)>\sigma(l)},
\end{align} and
\begin{align}\label{def:vectors2}
\begin{split}
    {\bf{n}}_{0,\rightarrow}:=\Big(n_{k\rightarrow l}\Big)_{(k,l):\ \sigma(k)>\sigma(l)},\\
    {\bf{n}}_{0,\leftarrow}:=\Big(n_{k\leftarrow l}\Big)_{(k,l):\ \sigma(k)>\sigma(l)},
\end{split}
\end{align} and ${\bf{n}_{0}}$ is a (natural) concatenation of $\bf{n}_{0,\rightarrow}$ and $\bf{n}_{0,\leftarrow}$; and $\partial {\bf{n}}_{0}=\emptyset$ signifies that, for any site $z,$
\begin{align}
  \sum_{\sigma(k)>\sigma(z)}  \Big(n_{k\rightarrow z}+n_{z\leftarrow k}\Big)=\sum_{\sigma(z)>\sigma(k)}  \Big( n_{k\rightarrow z}+n_{z\leftarrow k}\Big).
\end{align}

Next, we compute the quantity $F_{\beta,0,L}$ defined in \eqref{def:Finite}. 

Compute directly to obtain
\begin{align}\label{eq:fdeltal02}
\begin{split}
F_{\beta,0,L}(x,y)=&\frac{1}{(2\pi)^{\#\Lambda}}\int e^{-\beta H_{L,0} }S_x \cdot S_y\ \prod_{k\in \mathcal{L}^0}\ d\theta_k\\
= &\frac{1}{2}\Big(\sum_{
\partial {\bf{n}}_{0}=\{x\rightarrow y\}
}+\sum_{
\partial {\bf{n}}_{0}=\{y\rightarrow  x\}
} \Big)(\beta\bf{J})^{\bf{n}_{0}}\frac{1}{{\bf{n}_{0}}!}\\
=& F_{0,L}(x\rightarrow y)+F_{0,L}(y\rightarrow x)
\end{split}
\end{align} where  the terms $F_{0,L}(x\rightarrow y)$ and $F_{0,L}(y\rightarrow x)$ are defined naturally; $\partial {\bf{n}}=\{x\rightarrow  y\}$ signifies that, as a directed graph, $x$ is the only source and $y$ is the only sink, mathematically,
\begin{align}\label{eq:inOut}
    \partial {\bf{n}}=\{x\rightarrow  y\}\Longleftrightarrow \left\{
    \begin{array}{ll}
    \displaystyle\sum_{j}n_{k\rightarrow j}=\sum_{l}n_{l\leftarrow k}, &\ \text{at the site}\ k\not\in \{x,y\},\\
    \displaystyle\sum_{j}n_{x\rightarrow j}=\sum_{l}n_{l\leftarrow x}+1, & \ \text{at the site} \ x,\\
    \displaystyle\sum_{j}n_{y\rightarrow j}+1=\sum_{l}n_{l\leftarrow y}, & \ \text{at the site}\ y.
    \end{array}
    \right.
\end{align}

\subsection{Plus boundary condition}
We start with rewriting the definition of $\langle S_{x}\cdot S_{y}\rangle_{\beta,+,L}$ in \eqref{def:Finite}. The reason is that, as in \cite{ADS2015}, treating all the boundary sites as a so-called ghost site $\delta$ will simplify our analysis.

To prepare for that we rewrite the Hamiltonian. From \eqref{eq:Hamil}, 
\begin{align}
H_{L,+}=&-\sum_{k,l\in \mathcal{L}^0} J_{k,l}S_{k}\cdot S_{l}-\sum_{k\in \mathcal{L}^0,\ l\in \partial\mathcal{L}} J_{k,l}S_{k}\cdot \left(
\begin{array}{c}
1\\
0
\end{array}
\right)\nonumber\\
=&-\sum_{k,l\in \mathcal{L}^0} J_{k,l}S_{k}\cdot S_{l}-\sum_{k\in \mathcal{L}^0} J_{k,\delta}S_{k}\cdot \left(
\begin{array}{c}
1\\
0
\end{array}
\right),
\end{align} where $\delta$ is called the ghost site; and $J_{k,\delta}\equiv J_{\delta,k}$ is defined 
for any $k\in \mathcal{L}^{0}$,
\begin{align}
J_{\delta,k}:=\sum_{m\in \partial L}J_{m,k}.
\end{align}

Define a new Hamiltonian $H_{L,\delta}$, 
\begin{align}
H_{L,\delta}:=-\sum_{k,l\in \mathcal{L}^0} J_{k,l}S_{k}\cdot S_{l}-\sum_{k\in \mathcal{L}^0} J_{k,\delta}S_{k}\cdot S_{\delta}.
\end{align} 

Based on this, we are ready to rewrite the partition function. Since $$S_{k}\cdot S_{l}=(U S_k)\cdot (U S_l)$$ for any unitary $2\times 2$ matrix $U$, we have that,
\begin{align}
\langle S_{x}\cdot S_{y}\rangle_{\beta,+,L}=\frac{F_{\beta,+,L}(x,y)}{Z_{\beta,+,L}}=\frac{\tilde{F}_{\beta,\delta,L}(x,y)}{\tilde{Z}_{\beta,\delta,L}},
\end{align}
where, the quantities $\tilde{F}_{\delta,L}$ and $\tilde{Z}_{\delta,L}$ are defined as,
\begin{align*}
\tilde{F}_{\delta,L}(x,y):=&\frac{1}{(2\pi)^{\#\Lambda+1}}\int e^{ -\beta H_{L,\delta} } \prod_{k\in L^{0}\cup \delta} d^3 \theta_k,\\
\tilde{Z}_{\delta,L}:=&\frac{1}{(2\pi)^{\#\Lambda+1}}\int e^{ -\beta H_{L,\delta} } S_{x}\cdot S_{y}\prod_{k\in L^{0}\cup \delta} d^3 \theta_k.
\end{align*} 

What is left is similar to deriving \eqref{eq:fdeltal01} and \eqref{eq:fdeltal02}. Compute directly to find 
\begin{align}\label{eq:fdeltal1}
\begin{split}
\tilde{Z}_{\delta,L}=&\sum_{\partial({\bf{n}_{\delta,\rightarrow}+\bf{n}_{\delta,\leftarrow}})=\varnothing }\big(\beta\bf{J}\big)^{{\bf{n}}_{\delta,\rightarrow}+{\bf{n}}_{\delta,\leftarrow}}
\frac{1}{{\bf{n}_{\delta,\rightarrow}}!} \frac{1}{{\bf{n}_{\delta,\leftarrow}}!}\\
=&\sum_{\partial{\bf{n}_{\delta}}=\varnothing }\big(\beta\bf{J}\big)^{{\bf{n}}_{\delta}}
\frac{1}{{\bf{n}_{\delta}}!} ,
\end{split}
\end{align} and 
\begin{align}\label{eq:fdeltal2}
\begin{split}
\tilde{F}_{\delta,L}(x,y)
= &\frac{1}{2}\Big(\sum_{
\partial( {\bf{n}_{\delta,\rightarrow}+\bf{n}_{\delta,\leftarrow}})=\{x\rightarrow y\}
}+\sum_{
\partial( {\bf{n}_{\delta,\rightarrow}+\bf{n}_{\delta,\leftarrow}})=\{y\rightarrow  x\}
} \Big)(\beta\bf{J})^{\bf{n}_{\delta,\rightarrow}+\bf{n}_{\delta,\leftarrow}}\frac{1}{{\bf{n}_{\delta,\rightarrow}}!} \frac{1}{{\bf{n}_{\delta,\leftarrow}}! }\\
=&\frac{1}{2}\Big(\sum_{
\partial {\bf{n}_{\delta}}=\{x\rightarrow y\}
}+\sum_{
\partial {\bf{n}_{\delta}}=\{y\rightarrow  x\}
} \Big)(\beta\bf{J})^{\bf{n}_{\delta}}\frac{1}{{\bf{n}_{\delta}}!} \\
=& \tilde{F}_{\delta,L}(x\rightarrow y)+\tilde{F}_{\delta,L}(y\rightarrow x),
\end{split}
\end{align} where the terms $\tilde{F}_{\delta,L}(x\rightarrow y)$ and $\tilde{F}_{\delta,L}(y\rightarrow x)$ are naturally defined.


\section{Reformulation of Main Theorem \ref{MainTHM:first}}\label{sec:firstTHM}
A crucial step in proving Theorem \ref{MainTHM:first} is to show that, at the critical inverse temperature $\beta=\beta_c$, for any sites $x$ and $y$,
\begin{align}\label{eq:desP0}
    \langle S_x\cdot S_{y}\rangle_{+,\beta_c}=\langle S_x\cdot S_{y}\rangle_{0,\beta_c}.
\end{align} Its necessity and importance will be discussed after Proposition \ref{Prop:limitDxy}.

For preparation, we consider the following quantity
\begin{align}
    \langle S_x\cdot S_{y}\rangle_{L,+,\beta}-\langle S_x\cdot S_{y}\rangle_{L,0,\beta}=\frac{\tilde{F}_{\delta,L}(x,y)}{\tilde{Z}_{\delta,L}}-\frac{F_{0,L}(x,y)}{Z_{0,L}}
=&D_{1,x,y}(L)+D_{2,x,y}(L), \label{eq:differe}
\end{align}
where the terms $D_{1,x,y}(L)$ and $D_{2,x,y}(L)$ are defined as
\begin{align*}
D_{1,x,y}(L):=&\frac{\tilde{F}_{\delta,L}(x\rightarrow y)Z_{0,L}-F_{0,L}(y\rightarrow x)\tilde{Z}_{\delta,L}}{\tilde{Z}_{\delta,L}Z_{0,L}},\\
D_{2,x,y}(L):=&\frac{\tilde{F}_{\delta,L}(y\rightarrow x)Z_{0,L}-F_{0,L}(x\rightarrow y)\tilde{Z}_{\delta,L}}{\tilde{Z}_{\delta,L}Z_{0,L}}.
\end{align*}

The present paper is centered around devising a switching lemma to cancel the terms in $D_{1,x,y}(L)$ and $D_{2,x,y}(L)$, and prove the following result.
\begin{proposition}\label{Prop:limitDxy}
Suppose that Assumption \ref{assu:uniqClus} holds. Then, at the critical inverse temperature $\beta=\beta_c,$ for any fixed sites $x$ and $y$,
\begin{align}\label{eq:limDxy}
    \lim_{L\rightarrow \infty} D_{1,x,y}(L)=\lim_{L\rightarrow \infty} D_{2,x,y}(L)=0.
\end{align}
\end{proposition}
The proposition will be proved in Section \ref{sec:PropDxy}.

Proposition \ref{Prop:limitDxy} is important. Together with \eqref{eq:differe}, it implies that, for any fixed sites $x$ and $y$, at the critical inverse temperature $\beta_c,$
\begin{align}\label{eq:sxyp0}
    \Big\langle S_x\cdot S_y \Big\rangle_{+,\beta_c}=\Big\langle S_x\cdot S_y \Big\rangle_{0,\beta_c}.
\end{align} This identity is important for two reasons:
\begin{itemize}
    \item[(A)]
We can approximate $\Big\langle S_x\cdot S_y \Big\rangle_{0,\beta_c}$ by $\Big\langle S_x\cdot S_y \Big\rangle_{L, 0,\beta}$ with $L\gg 1$ and $\beta<\beta_c$ by using the following identities
\begin{align}\label{eq:ChFocus}
    \lim_{L\rightarrow \infty}\lim_{\beta \nearrow \beta_c}\Big\langle S_x\cdot S_y \Big\rangle_{L,0,\beta}=\lim_{\beta \nearrow \beta_c}\lim_{L\rightarrow \infty}\Big\langle S_x\cdot S_y \Big\rangle_{L,0,\beta}=\Big\langle S_x\cdot S_y \Big\rangle_{0,\beta_c};
\end{align} 
\item[(B)]
When $\beta< \beta_c$, one can use \eqref{eq:compare0toPeriodic} to make the technique of Gaussian domination applicable, see Lemma \ref{LM:highTem} below. For the details, see the proof of Theorem \ref{THM:second} below.
\end{itemize}

Assuming Proposition \ref{Prop:limitDxy}, we are ready to prove Main Theorem \ref{MainTHM:first}.
\subsection{Proving Main Theorem \ref{MainTHM:first} by assuming Proposition \ref{Prop:limitDxy}}
Recall that our goal is to prove that $\Big\langle \cos(\theta_0)\Big\rangle_{+,\beta_c}=0$.

By the well-known results, at any inverse temperature $\beta>0,$
\begin{align}\label{eq:4IdIn}
    0\leq \big\langle \cos(\theta_0)\big\rangle_{+,\beta}^2=\big\langle \cos(\theta_x) \big\rangle_{+,\beta}  \big\langle \cos(\theta_y) \big\rangle_{+,\beta}\leq  \big\langle \cos(\theta_x) \cos(\theta_y) \big\rangle_{+,\beta}\leq \big\langle S_x\cdot S_y \big\rangle_{+,\beta},
\end{align} where, in the second step, we used the identity, for any fixed site $y$,
\begin{align*}
    \Big\langle \cos(\theta_0)\Big\rangle_{+,\beta}=  \Big\langle \cos(\theta_y) \Big\rangle_{+,\beta};
\end{align*} and in the third step we used Griffiths inequality, see e.g. \cite{Griffiths1967CorrelationsI, benassi2017correlation,FV2018}; in the last step, we used the identity $S_x\cdot S_y=\cos(\theta_x)\cos(\theta_y)+\sin(\theta_x)\sin(\theta_y)$ and that $\Big\langle \sin(\theta_x) \sin(\theta_y)\Big\rangle_{+,\beta}\geq 0$, see e.g. \cite{benassi2017correlation}.

We continue to control $\big\langle S_x\cdot S_y \big\rangle_{+,\beta}$ in \eqref{eq:4IdIn}, but at $\beta=\beta_c$. Here the important \eqref{eq:sxyp0}, together with \eqref{eq:4IdIn}, implies that, 
\begin{align}\label{eq:PlusToFree}
 \text{for any fixed sites}\ x \ \text{and}\ y,\   0\leq \Big\langle \cos(\theta_0)\Big\rangle_{+,\beta_c}^2\leq \Big\langle S_x\cdot S_y \Big\rangle_{0,\beta_c}.
\end{align}
On the other hand, Theorem \ref{THM:second} implies that there exist sequences $\big\{x_n\big\}_{n=1}^{\infty}$ and $\big\{y_n\big\}_{n=1}^{\infty}$ such that $$
\lim_{n\rightarrow \infty}\Big\langle S_{x_n}\cdot S_{y_n} \Big\rangle_{0,\beta_c}=0.$$ 
This and \eqref{eq:PlusToFree} imply the desired result
\begin{align}
    \Big\langle \cos(\theta_0)\Big\rangle_{+,\beta_c}=0.
\end{align}



\section{Proof of Theorem \ref{THM:second}}\label{sec:secondTHM}
To make the technique of Gaussian domination applicable, we consider the periodic boundary condition on a $3-$dimensional periodic lattice $\Omega_L:=\Big(\mathbb{Z}/(2L\mathbb{Z})\Big)^3$ with Hamilton $\tilde{H}_{L}$ defined as
\begin{align}\label{def:Periodic}
    \tilde{H}_{L}:=-\frac{1}{6}\sum_{k,l\in \Omega_{L},\ |k-l|=1}  S_{k} \cdot S_{l}
\end{align} where $S_k\in \mathbb{S}^1$ are functions periodic in $k\in \mathbb{Z}^3$, specifically, for any $j\in 2L\mathbb{Z}^3$,
\begin{align*}
    S_k=&S_{k+j}.
\end{align*}

In what follows we use the notation $\langle\  \cdot\ \rangle_{L,P,\beta}$ to denote the expectation with the periodic boundary condition.

Studying the periodic lattice is helpful for the present problem: Apply the Ginibre's inequality, see \cite{Ginibre1970}, to obtain, 
\begin{align}\label{eq:compare0toPeriodic}
    0\leq \Big\langle S_x\cdot S_y\Big\rangle_{L,0,\beta}\leq \Big\langle S_x \cdot S_y\Big\rangle_{L,P,\beta},
\end{align} provided that $L$ is large enough so that $|x|+|y|<L.$

What is left is to control $\Big\langle S_x \cdot S_y\Big\rangle_{L,P,\beta}$ by the technique of Gaussian domination. 

Define a function 
\begin{align}\label{def:Gxy}
    G(x,y):=\frac{1}{(2\pi)^3} \int_{[-\pi,\pi]^3} \frac{e^{ik\cdot (x-y)}}{1-\hat{J}(k)} d^3 k,
\end{align} where, by \eqref{def:Periodic}, $J(l)=\left\{
\begin{array}{cc}
     \frac{1}{6} & \text{if}\ |l|=1,\\
     0& \text{otherwise,}
\end{array} 
\right.$ and thus $$\hat{J}(k)=\frac{1}{3} \big(\cos(k_1)+\cos(k_2)+\cos(k_3)\big)$$ for $k=(k_1,k_2,k_3)\in \mathbb{R}^3$. Here $G(x,y)$ is well defined because $\frac{1}{1-\hat{J}(k)}>0$ and it has only one singularity at $|k|=0$, and this singularity is of the form $|k|^{-2}(3+o(1))$ and, thus, is integrable.

The result is the following:
\begin{lemma}\label{LM:highTem} When $\beta< \beta_c,$
\begin{align}
     \lim_{L\rightarrow \infty} \frac{1}{|B_n|^2}\sum_{x,y\in B_n}\Big\langle S_x \cdot S_y\Big\rangle_{L,P,\beta}\leq \frac{1}{2\beta}\frac{1}{|B_n|^2}\sum_{x,y\in B_n}G(x,y).
\end{align} Especially, $G$ is independent of $\beta$ and
\begin{align}\label{eq:LimGXY}
    \lim_{|x-y|\rightarrow \infty} G(x,y)=0.
\end{align}
\end{lemma}
This is based on the works of Fr\"{o}hlich, Spencer, etc in \cite{FSS76,FILS78}, see also \cite{Biskup09,FV2018}. \eqref{eq:LimGXY} follows from \eqref{def:Gxy} and the fact that, as $|x-y|$ becomes large, the oscillation renders the integral small.

We continue to prove Theorem \ref{THM:second}

This and \eqref{eq:compare0toPeriodic} imply the desired property, for any $\beta<\beta_c,$
\begin{align}\label{eq:ShortRan}
 \lim_{L\rightarrow \infty} \frac{1}{|B_n|^2}\sum_{x,y\in B_n}\Big\langle S_{x}\cdot S_y\Big\rangle_{L,0,\beta}\leq \frac{1}{2\beta}\frac{1}{|B_n|^2}\sum_{x,y\in B_n}G(x,y).
\end{align}

Next, we increase $\beta$ and $L$ to $\beta_c$ and $\infty$, respectively. The limit exists by the following proposition.
\begin{proposition}
    For any fixed sites $x$, $y\in \mathbb{Z}^3$, $\langle S_x\cdot S_y\rangle_{L,0,\beta}$ is an increasing function in both $L$ and $\beta$ variables, and
\begin{align}\label{eq:TwoLimit}
\lim_{\beta \nearrow \beta_c}\lim_{L\rightarrow \infty}\langle S_x\cdot S_y\rangle_{L,0,\beta}=\lim_{L\rightarrow \infty}\lim_{\beta \nearrow \beta_c}\langle S_x\cdot S_y\rangle_{L,0,\beta}=\langle S_x\cdot S_y\rangle_{0,\beta_c}.
    \end{align}
\end{proposition}
\begin{proof}
The proof is well-known, thus we only state some key steps. Here we need two key observations: by Ginibre's inequality, see e.g. \cite{Ginibre1970,Griffiths1967CorrelationsI, benassi2017correlation, FV2018},
$\langle S_x\cdot S_y\rangle_{L,0,\beta}$ is an increasing function in the variables $L$ and $\beta;$ for any fixed $L$, $\langle S_x\cdot S_y\rangle_{L,0,\beta}$ is uniformly continuous in the variable $\beta.$
\end{proof}

We are ready to prove Theorem \ref{THM:second}. 

\eqref{eq:TwoLimit}, \eqref{eq:ShortRan} and \eqref{eq:LimGXY} imply the desired result
\begin{align}
   \lim_{n\rightarrow \infty}\frac{1}{|B_n|^2}\sum_{x,y\in B_n}\langle S_x\cdot S_y\rangle_{0,\beta_c}  =0.
\end{align}

The proof is complete.

\section{Proof of Proposition \ref{Prop:limitDxy}}\label{sec:PropDxy}
Our general strategy is to show that, under Assumption \ref{assu:uniqClus}, when $L\rightarrow \infty$, a properly defined switching lemma will help us to cancel ``almost all" the terms in $D_{1,x,y}$ and $D_{2,x,y}$.

We start with establishing a switching lemma.

In preparation for that, we define
\begin{align}
    D_{1,x,y}=E_{1,x\rightarrow y}-E_{2, y\rightarrow x}
\end{align} where $E_{1,x\rightarrow y}$ and $E_{2, y\rightarrow x}$ are defined as
\begin{align*}
   E_{1,x\rightarrow y}:= &\frac{\tilde{F}_{\delta,L}(x\rightarrow y)Z_{0,L}}{\tilde{Z}_{\delta,L}Z_{0,L}},\\
   E_{2, y\rightarrow x}:=&   \frac{ \tilde{Z}_{\delta,L}\ F_{0,L}(y\rightarrow x)}{\tilde{Z}_{\delta,L}Z_{0,L}}.
\end{align*}

To rewrite $E_{1,x\rightarrow y}$ we use the expressions of $\tilde{F}_{\delta,L}(x\rightarrow y)$ and $Z_{0,L}$ in \eqref{eq:fdeltal01}, \eqref{eq:fdeltal02}, \eqref{eq:fdeltal1} and \eqref{eq:fdeltal2}, to find
\begin{align}\label{eq:FdLZL1}
\begin{split}
 E_{1,x\rightarrow y}&=\frac{1}{2 \tilde{Z}_{\delta,L}Z_{0,L} }\sum_{
    \partial{\bf{n}}_{\delta}=\{x\rightarrow y\},\
    \partial{{\bf{n}}_0}=\emptyset
    } (\beta{\bf{J}})^{{\bf{n}}_{\delta}+ {\bf{n}}_{0}}\frac{1}{({\bf{n}}_{\delta})!}
    \frac{1}{(\bf{n}_{0})!}\\
    &=\sum_{{\bf{n}}({\bf{n}}_{\delta},{\bf{n}}_{0}):\ 
    \partial{\bf{n}}_{\delta}=\{x\rightarrow y\};\ 
    \partial{\bf{n}}_{0}=\emptyset
    } V_{\bf{J},\bf{n}({\bf{n}_{\delta},\bf{n}_{0}})} \Big|\mathcal{S}_{\delta,{\bf{n}}({\bf{n}_{\delta},\bf{n}_{0}}), L}(x\rightarrow y)\Big|,
\end{split}
\end{align} where 
$\bf{J}$, $\bf{n}_{\delta}$ and $\bf{n}_{0}$ are vectors defined in \eqref{def:vectors1} and \eqref{def:vectors2}; $\bf{n}({\bf{n}}_\delta,{\bf{n}}_0)$ is a nonnegative-integer-valued vector defined as follows: for any two nearest neighbors $a,\ b\in \mathbb{Z}^3$ on the lattice, suppose that $\bf{n}_{\delta}$ and $\bf{n}_0$ contain $n_{\delta,1}$ and $n_{0,1}$ $a\rightarrow b$ edges and $n_{\delta,2}$ and $n_{0,2}$ $b\rightarrow a$ edges, respectively, then
\begin{align}\label{def:nndeltan0}
    {\bf{n}({\bf{n}}_\delta,{\bf{n}}_0)}\Big|_{(a,b)}:=n_{\delta,1}+n_{0,1}+n_{\delta,2}+n_{0,2};
\end{align}
and $\mathcal{S}_{\delta,{\bf{n}}({\bf{n}_{\delta},\bf{n}_{0}}), L}(x\rightarrow y)$ is a set of directed graphs defined as
\begin{align}\label{def:sdnn}
\begin{split}
    \mathcal{S}_{\delta,{\bf{n}}({\bf{n}_{\delta},\bf{n}_{0}}), L}(x\rightarrow y):=\Big\{ \Big(\mathcal{G}_{\bf{n}_{\delta}}(x\rightarrow y),\mathcal{G}_{\bf{n}_0}\Big)\   \Big| & \ \mathcal{G}_{\bf{n}_{\delta}},\mathcal{G}_{\bf{n}_0}\ \text{are graphs formed by} \ \bf{n}_{\delta}\ \text{and}\ \bf{n}_0, \\
    &\text{and}\
    \partial{\bf{n}}_{\delta}=\{x\rightarrow y\},\ 
    \partial{\bf{n}}_{0}=\emptyset\Big\},
\end{split}
\end{align}
and the integer $$\Big|\mathcal{S}_{\delta,{\bf{n}}({\bf{n}_{\delta},\bf{n}_{0}}),L }(x\rightarrow y)\Big|=\frac{({\bf{n}({\bf{n}}_\delta,{\bf{n}}_0)} ) !}{(\bf{n}_{\delta})!\ (\bf{n}_0)!}$$ is interpreted as the total number of directed graphs in the set: at each nearest neighboring sites $a$ and $b$, there are in total $\frac{(n_{\delta,1}+n_{0,1}+n_{\delta,2}+n_{0,2})!}{(n_{\delta,1})!\ (n_{0,1})!\ (n_{\delta,2})!\ (n_{0,2})!}$ ways of distributing $n_{\delta,1}+n_{0,1}+n_{\delta,2}+n_{0,2}$ edges; and the positive number
$$V_{\bf{J},\bf{n}({\bf{n}_{\delta},\bf{n}_{0}})}:=\frac{1}{2 \tilde{Z}_{\delta,L}Z_{0,L} }(\beta{\bf{J}})^{{\bf{n}}({\bf{n}}_\delta,{\bf{n}}_0)} \frac{1}{ (\bf{n}({\bf{n}}_\delta,{\bf{n}}_0))! }$$ is considered the weight assigned to each graph in the set $ \mathcal{S}_{\delta,{\bf{n}}({\bf{n}_{\delta},\bf{n}_{0}}),L}(x\rightarrow y)$.

Similarly, for $E_{2,y\rightarrow x}$,
\begin{align}\label{eq:FdLZL2}
\begin{split}
E_{2, y\rightarrow x}
   &=\frac{1}{2 \tilde{Z}_{\delta,L}Z_{0,L} }\sum_{
    \partial{\bf{n}}_{0}=\{y\rightarrow x\},\ 
    \partial{\bf{n}_{\delta}}=\emptyset
    } (\beta{\bf{J}})^{{\bf{n}}_{\delta}+{\bf{n}}_{0}}\frac{1}{({\bf{n}}_{\delta})!}
    \frac{1}{(\bf{n}_{0})!}\\
    &=\sum_{ {\bf{n}}({\bf{n}_{\delta},\bf{n}_{0}}):\ 
    \partial{\bf{n}}_{0}=\{y\rightarrow x\},\
    \partial{\bf{n}}_{\delta}=\emptyset
    } V_{\bf{J},{\bf{n}}({\bf{n}_{\delta},\bf{n}_{0}}) } \Big|\mathcal{S}_{0,{\bf{n}}({\bf{n}_{\delta},\bf{n}_{0}}),L}(y\rightarrow x)\Big|.
\end{split}
\end{align} 
Here the terms are interpreted similarly to the corresponding ones in \eqref{eq:FdLZL1}.

To illustrate the difficulties, we present a preliminary version of path switching.

For a fixed $\delta-$site-avoiding $x\rightarrow y$ path $\mathcal{P}=\mathcal{P}(x\rightarrow y)$, we define two sets of graphs
\begin{align*}
    A_{x\rightarrow y, \mathcal{P}}:=&\Big\{  \Big(\Phi_{\delta}(x\rightarrow y),\Phi_{0}\Big) \ \Big| \ \mathcal{P}\subset \Big(\Phi_{\delta}(x\rightarrow y),\Phi_{0}\Big) \in \bigcup_{{\bf{n}}({\bf{n}_{\delta},\bf{n}_{0}})  }\mathcal{S}_{\delta,{\bf{n}}({\bf{n}_{\delta},\bf{n}_{0}}) , L}(x\rightarrow y)\Big\},\\
    B_{y\rightarrow x, \bar{\mathcal{P}} }:=&\Big\{  \Big(\Psi_{\delta},\Psi_{0}(y\rightarrow x)\Big)\ \Big|\ \bar{\mathcal{P}}\subset \Big(\Psi_{\delta},\Psi_{0}(y\rightarrow x)\Big) \in \bigcup_{{\bf{n}}({\bf{n}_{\delta},\bf{n}_{0}})  }\mathcal{S}_{0,{\bf{n}}({\bf{n}_{\delta},\bf{n}_{0}}) , L}(y\rightarrow x)\Big\}.
\end{align*} Here $\bar{\mathcal{P}}$ is the inverse of the directed path $\mathcal{P}.$

Given these two sets of graphs, we can define a bijection, denoted by
\begin{align}\label{def:Pswitch}
   A_{x\rightarrow y, \mathcal{P}} \xrightarrow{\mathcal{P}} B_{y\rightarrow x, \bar{\mathcal{P}} },
\end{align}
by switching the chosen path $\mathcal{P}$. The procedure was used in \cite{van2021elementary} and works as following:
For a fixed $\Big(\Phi_{\delta}(x\rightarrow y),\Phi_{0}\Big)\in A_{x\rightarrow y, \mathcal{P}}$ we switch the path $\mathcal{P}$ to obtain a new graph $\Big(\Psi_{\delta},\Psi_{0}(y\rightarrow x)\Big)\in B_{y\rightarrow x, \bar{\mathcal{P}}}$
such that,
\begin{align}\label{eq:switch}
    \begin{split}
        \Psi_{\delta}:=&\Big(\Phi_{\delta}\backslash \mathcal{P}\Big)\cup \Big(\overline{\Phi_{0}\cap \mathcal{P}}\Big);\\
        \Psi_{0}(y\rightarrow x):=& \Big(\Phi_{0}\backslash \mathcal{P}\Big)\cup \Big(\overline{\Phi_{\delta}\cap \mathcal{P}}\Big).
    \end{split}
\end{align}

Importantly, this mapping is reversible and is between graphs of equal weights.

In summary, we obtain the following results:
\begin{lemma}\label{LM:biBeforePairing}
    For a fixed $\delta-$site-avoiding $x\rightarrow y$ path $\mathcal{P}$, the switching defined in \eqref{eq:switch} establishes a bijection between $A_{x\rightarrow y, \mathcal{P}}$ and $B_{y\rightarrow x, \bar{\mathcal{P}}}$. And the mapping $\xrightarrow{\mathcal{P}}$ is between graphs of equal weight.
\end{lemma}

Now we discuss the limitation. In brief, the switching defined in \eqref{def:Pswitch} only works occasionally.
\begin{enumerate}
\item[(A1)] The first problem is that some graphs in $\bigcup_{{\bf{n}}({\bf{n}_{\delta},\bf{n}_{0}}) }\mathcal{S}_{\delta,{\bf{n}}({\bf{n}_{\delta},\bf{n}_{0}}), L}(x\rightarrow y)$ might not contain a ``switchable" $x\rightarrow y$ path $\mathcal{P}(x\rightarrow y)$, for example, every $\mathcal{P}(x\rightarrow y)$ path passes through the $\delta$-site. Thus, there is no path to switch to begin with! 

\item[(A2)] The second one is more fundamental. In brief, even if every graph has at least a path for us to switch, we still have difficulties finding a bijection, partly because we do not have a conservation law as \eqref{eq:conserv}.

It is easy to construct a ``bad" example: there exist two distinct $\delta-$site-avoiding $x\rightarrow y$-paths $\mathcal{P}$ and $\mathcal{Q}$, and two graphs $ \mathcal{G}=\Big(\mathcal{G}_{\delta }(x\rightarrow y),\mathcal{G}_0\Big) $ and $\mathcal{F}=\Big(\mathcal{F}_{\delta }(x\rightarrow y),\mathcal{F}_0\Big)$ satisfying
\begin{align}\label{eq:adverse1}
\begin{split}
\mathcal{P}\subset  \mathcal{G}_{\delta}(x\rightarrow y)\cup \mathcal{G}_0, \ & \text{but}\ 
\mathcal{P}\not\subset \mathcal{F}_{\delta}(x\rightarrow y)\cup \mathcal{F}_0;\\
\mathcal{Q}\not\subset  \mathcal{G}_{\delta}(x\rightarrow y)\cup \mathcal{G}_0, \ & \text{but}\ 
\mathcal{Q}\subset \mathcal{F}_{\delta}(x\rightarrow y)\cup \mathcal{F}_0,
\end{split}
\end{align} such that, after switching the paths $\mathcal{P}$ and $\mathcal{Q}$ in $\mathcal{G}$ and $\mathcal{F}$ respectively to obtain two new graphs, denoted by $\mathcal{G}^{(\mathcal{P})}$ and $\mathcal{F}^{(\mathcal{Q})}$, the following adverse situation occurs:
\begin{align}\label{eq:adverse2}
   \mathcal{G}^{(\mathcal{P})}= \mathcal{F}^{(\mathcal{Q})}.
\end{align} Consequently, a simple generalization of switching lemma for an undirected graph does not work here!
\end{enumerate}

To overcome the first difficulty, we need Assumption \ref{assu:uniqClus}. For the detail see the discussions in Section \ref{sec:assu} below.

To overcome the second difficulty, we pair the edges in \eqref{eq:InEqualOut}-\eqref{eq:pairing} below. It is indeed an improvement, but it also creates a new difficulty.
This forces us to refine the pairing in Subsection \ref{sub:renormalize}.

\subsection{Edge pairing: a preliminary version}\label{sub:pairing}
Edge-pairing helps, at least partly, to overcome the second difficulty discussed above. It is not difficult to understand intuitively, because, after the edges are paired, any graph in $\bigcup_{{\bf{n}}({\bf{n}_{\delta},\bf{n}_{0}})  }\mathcal{S}_{\delta,{\bf{n}}({\bf{n}_{\delta},\bf{n}_{0}}), L}(x\rightarrow y)$ must have a paired $x\rightarrow y$-path. Provided that it does not pass through the $\delta-$site, then we switch this paired path to establish a bijection, and the adverse situations in \eqref{eq:adverse1} and \eqref{eq:adverse2} will not appear.

To prepare for pairing, we observe that, for any fixed site 
\begin{align}\label{eq:zNOTxy}
z\not\in\{x,y\},
\end{align}
the number of incoming edges is $\displaystyle\sum_{j:\sigma(z)>\sigma(j)}(n_{\delta,z\leftarrow j,}+n_{0,z\leftarrow j})$, and that of outgoing ones is $\displaystyle\sum_{j:\sigma(z)>\sigma(j)}(n_{\delta,z\rightarrow j}+n_{0,z\rightarrow j})$. These two numbers are equal: by \eqref{eq:inOut} for the $0-$boundary condition, which also holds for the plus-boundary condition, 
\begin{align}\label{eq:InEqualOut}
    \displaystyle\sum_{j:\sigma(z)>\sigma(j)}(n_{\delta,z\leftarrow j}+n_{0,z\leftarrow j})=\displaystyle\sum_{j:\sigma(z)>\sigma(j)}(n_{\delta,z\rightarrow j}+n_{\delta,z\rightarrow j}).
\end{align}
Therefore, at any site $z\not\in \{x,y\}$, there are 
\begin{align}\label{def:Psi}
\begin{split}
   \Psi_{\bf{n}({\bf{n}}_\delta,{\bf{n}_0})}(z) &:=\Big(\sum_{j}(n_{\delta,z\rightarrow j}+n_{0,z\rightarrow j})\Big)!\\
   &=\Big(\frac{1}{2}\sum_{j}(n_{\delta,z\rightarrow j}+n_{\delta,j\rightarrow z}+n_{0,z\rightarrow j}+n_{0,j\rightarrow z})\Big)!
\end{split}
\end{align} ways of pairing the incoming edges with the outgoing ones. Recall that $\bf{n}({\bf{n}}_\delta,{\bf{n}}_0)$ was defined in \eqref{def:nndeltan0}.

When $z=x$ or $z=y$, we choose not to pair the edges to avoid messy combinatorics. This leads to a minor complication: when ``switching", we not only switch a paired $x\rightarrow y$-path, but switch all the paired $x\rightarrow y$, $y\rightarrow x$ paths, and $x\rightarrow x$, $y\rightarrow y$ loops, see e.g. \eqref{eq:pairedPSwi} below.

Consequently, after pairing, \eqref{eq:FdLZL1} and \eqref{eq:FdLZL2} take new forms
\begin{align}
\begin{split}\label{eq:pairing}
    \frac{\tilde{F}_{\delta,L}(x\rightarrow y)Z_{0,L}}{\tilde{Z}_{\delta,L}Z_{0,L}}&=\sum_{ \bf{n}({\bf{n}}_\delta,{\bf{n}}_0):\ \partial{\bf{n}}_{\delta}=\{x\rightarrow y\},\ \partial{\bf{n}}_0=\emptyset } 
    W_{ \bf{n}({\bf{n}}_\delta,{\bf{n}}_0) ,L }\ \Big|\Gamma_{\delta, \bf{n}({\bf{n}}_\delta,{\bf{n}}_0) ,L }(x\rightarrow y)\Big|;\\
    \frac{F_{0,L}(y\rightarrow x)\tilde{Z}_{\delta,L}}{\tilde{Z}_{\delta,L}Z_{0,L}}&=\sum_{\bf{n}({\bf{n}}_\delta,{\bf{n}}_0):\ \partial{\bf{n}}_{0}=\{y\rightarrow x\},\ \partial{\bf{n}}_{\delta}=\emptyset } W_{ \bf{n}({\bf{n}}_\delta,{\bf{n}}_0) ,L }\ \Big|\Gamma_{0, \bf{n}({\bf{n}}_\delta,{\bf{n}}_0) ,L }(y\rightarrow x)\Big|,
\end{split}
\end{align} where $W_{ \bf{n}({\bf{n}}_\delta,{\bf{n}}_0) ,L }$ is considered the weight assigned to each paired graph formed by ${\bf{n}}_{\delta}$ and ${\bf{n}}_0$:
\begin{align}
    W_{ \bf{n}({\bf{n}}_\delta,{\bf{n}}_0)  ,L}:=\frac{1}{2}(\beta{\bf{J}})^{ \bf{n}({\bf{n}}_\delta,{\bf{n}}_0) }\frac{1}{ \bf{n}({\bf{n}}_\delta,{\bf{n}}_0)! }\frac{1}{\Psi_{ \bf{n}({\bf{n}}_\delta,{\bf{n}}_0) }}\frac{1}{\tilde{Z}_{\delta,L}Z_{0,L}},
\end{align}
$\Psi_{ \bf{n}({\bf{n}}_\delta,{\bf{n}}_0) }$ is a positive integer defined in terms of $\Psi_{ \bf{n}({\bf{n}}_\delta,{\bf{n}}_0) }(z)$ in \eqref{def:Psi},
\begin{align*}
\Psi_{ \bf{n}({\bf{n}}_\delta,{\bf{n}}_0) }:=&\prod_{z\not\in\{ x,\ y\}}\Psi_{ \bf{n}({\bf{n}}_\delta,{\bf{n}}_0) }(z);
\end{align*}  $\Gamma_{\delta,  \bf{n}({\bf{n}}_\delta,{\bf{n}}_0) ,L  }(x\rightarrow y)$ is a set of paired graph defined as
\begin{align*}
    \Gamma_{\delta, \bf{n}({\bf{n}}_\delta,{\bf{n}}_0) , L}(x\rightarrow y):=\Big\{ & 
    G_{{\bf{n}}_{\delta},{\bf{n}}_{0},L}(x\rightarrow y)\ \text{is a paired graph formed by} \ \bf{n}_{\delta} \ \text{and}\ \bf{n}_0;\\ 
    & \hspace{0.5 cm} \partial{\bf{n}}_{\delta}=\{x\rightarrow y\};\ \partial{\bf{n}}_0=\emptyset
    \Big\}
\end{align*} and
\begin{align*}
    \Big|\Gamma_{\delta, \bf{n}({\bf{n}}_\delta,{\bf{n}}_0) ,L }(x\rightarrow y)\Big|:=&\text{the total number of paired graphs in}\ \Gamma_{\delta, \bf{n}({\bf{n}}_\delta,{\bf{n}}_0) , L}(x\rightarrow y)\\
    =&\sum_{{\bf{n}}({\bf{n}}_\delta,{\bf{n}}_0) :\ \partial {\bf{n}}_\delta=\{x\rightarrow y\},\ \partial {\bf{n}}_0=\emptyset   } \frac{ \Big(\bf{n}(\bf{n}_{\delta},\bf{n}_{0}) \Big)! }{(\bf{n}_{\delta})!\ (\bf{n}_0)!}\ \Psi_{ \bf{n}({\bf{n}}_\delta,{\bf{n}}_0) }.
\end{align*} 

$\Gamma_{0, \bf{n}({\bf{n}}_\delta,{\bf{n}}_0) ,L }(y\rightarrow x)$ in the second line of \eqref{eq:pairing} is interpreted similarly, and
$$\Big|\Gamma_{0, \bf{n}({\bf{n}}_\delta,{\bf{n}}_0) ,L }(y\rightarrow x)\Big|=\sum_{ {\bf{n}}({\bf{n}}_\delta,{\bf{n}}_0):\ \partial {\bf{n}}_\delta=\emptyset,\ \partial {\bf{n}}_0=\{y\rightarrow x\}    
 }\frac{ \Big(\bf{n}(\bf{n}_{\delta},\bf{n}_{0}) \Big)! }{(\bf{n}_{\delta})!\ (\bf{n}_0)!}\Psi_{ \bf{n}({\bf{n}}_\delta,{\bf{n}}_0)  }.$$

It is not difficult to prove that, when $L\rightarrow \infty$, the limit exists for paired graphs in \eqref{eq:pairing} with weights. We choose to skip the details here.

The following definition of a projective map will facilitate later discussions. A (unpaired) graph might generate different paired ones. For identification we define a projection, denoted by $\Pi$, to project a paired graph to its underlying unpaired one. 
\begin{definition}\label{def:projective}
    $\Pi$ is the projective map that maps a paired graph to the corresponding unpaired one.
\end{definition}

Now we discuss the advantages of pairing and its limitations.
 
Since $x$ and $y$ are the only source and sink respectively, any paired graph generated by $\frac{\tilde{F}_{\delta,L}(x\rightarrow y) Z_{0,L}}{{\tilde{Z}_{\delta,L}Z_{0,L}}}$ in \eqref{eq:pairing} must contain at least one paired $x\rightarrow y$ path. 

Before ``switching", we define a notion of a switchable graph $P(x\rightarrow y)$. Here we emphasize that we switch a graph $P(x\rightarrow y)$, instead of a path. The reason is that: Because we do not pair edges at the sites $x$ and $y$, we switch all the paths and loops connecting $x$ and $y$ in order to avoid some messy combinatorics. 
\begin{definition}\label{def:switchable}
Let $P=P(x\rightarrow y)$ be a fixed set of paired paths $x\rightarrow y$, $y\rightarrow x$, and paired loops $x\rightarrow x$ and $y\rightarrow y$, except that, at the sites $x$ and $y$, the edges are not paired. We say $P$ is a switchable paired $x\rightarrow y$-graph if these paths and loops avoid the $\delta$-site.
\end{definition}

To see the advantage of edge pairing, we fix an switchable paired graph $P=P(x\rightarrow y)$, and define two subsets of $\displaystyle\bigcup_{ {\bf{n}}({\bf{n}}_{\delta},{\bf{n}}_{0})}\Gamma_{\delta, {\bf{n}}({\bf{n}}_{\delta},{\bf{n}}_{0}) , L}$ and $\displaystyle\bigcup_{ {\bf{n}}({\bf{n}}_{\delta},{\bf{n}}_{0})  }\Gamma_{0, {\bf{n}}({\bf{n}}_{\delta},{\bf{n}}_{0}) ,L }$:
\begin{align*}
    \tilde{A}_{x\rightarrow y, P}:=\Big\{G_{\delta}(x\rightarrow y)&\Big|\ G_{\delta}(x\rightarrow y)\in \bigcup_{ {\bf{n}}({\bf{n}}_{\delta},{\bf{n}}_{0}) }\Gamma_{\delta, {\bf{n}}({\bf{n}}_{\delta},{\bf{n}}_{0}) , L}(x\rightarrow y),\\ & \text{ and all the paired paths and loops connected to} \ x\ \text{or}\ y\ \text{are in}\ P\Big\};
    \\
    \tilde{B}_{y\rightarrow x, \bar{P}}:=\Big\{G_{0}(y\rightarrow x)&\Big|\ G_{0}(y\rightarrow x) \in \bigcup_{ {\bf{n}}({\bf{n}}_{\delta},{\bf{n}}_{0}) }\Gamma_{0, {\bf{n}}({\bf{n}}_{\delta},{\bf{n}}_{0}),L }(y\rightarrow x),\\ &\text{and all the paired paths and loops connected to} \ x\ \text{or}\ y\ \text{are in}\ \bar{P}\Big\}.
\end{align*} Here $\bar{P}$ is the reverse of $P.$

Between these two sets we switch the chosen paired graph $P$ to establish a bijection between $\tilde{A}_{x\rightarrow y, P}$ and $\tilde{B}_{y\rightarrow x, \bar{P}}$, denoted by
\begin{align}\label{eq:pairedPSwi}
 \tilde{A}_{x\rightarrow y, P}  \xrightarrow{P} \tilde{B}_{y\rightarrow x, \bar{P}}.
\end{align}
by the same procedure for switching described in \eqref{eq:switch}. 

It is important to note that the switching is between graphs of equal weights.

Based on these, it is easy to obtain the following results.
\begin{lemma}\label{LM:biFirstPairing}
    Between sets $\displaystyle\bigcup_{\text{ switchable}\ P}\tilde{A}_{x\rightarrow y, P}$ and $\displaystyle\bigcup_{ \text{switchable}\ P}\tilde{B}_{y\rightarrow x, \bar{P}}$, the naturally defined switching is a bijection between graphs of equal weight.
\end{lemma}

However, the procedure described above has some limits because it is not always useful.

The problem is that the sets $\displaystyle\bigcup_{\text{ switchable}\ P}\tilde{A}_{x\rightarrow y, P}$ and $\displaystyle\bigcup_{ \text{switchable}\ P}\tilde{B}_{y\rightarrow x, \bar{P}}$ might be only subsets of all the considered paired graphs in \eqref{eq:pairing} for two reasons:
\begin{itemize}
    \item[(B1)] The first one is the same as (A1) discussed earlier: even before pairing, some graphs do not contain a $\delta-$site-avoiding $x\rightarrow y$ path.
    \item[(B2)] The second is that, even if the underlying unpaired graph has a usable $x\rightarrow y$ path, it might happen that all the paired $x\rightarrow y$-paths pass through the $\delta$-site, and thus, the paired graph does not have a switchable path. 
    
\end{itemize}

We will circumvent the first difficulty in the Subsection \ref{sec:assu}, based on Assumption \ref{assu:uniqClus}. The second difficulty will be addressed in Subsection \ref{sub:renormalize}.


\subsection{Implication of Assumption \ref{assu:uniqClus}: Resolution of Difficulties [A1] and [B1]}\label{sec:assu}
We need Assumption \ref{assu:uniqClus} because it implies the following result: 
\begin{proposition}\label{prop:uniq}
Suppose Assumption \ref{assu:uniqClus} holds. 

Then, for both free- and plus-boundary conditions, when $L\rightarrow \infty,$ almost surely, every directed graph $\mathcal{G}$ generated by the random path representation is a union of finite loops. 

More precisely, with probability 1, $\mathcal{G}$ is a union of finite loops
\begin{align}\label{eq:unionLoop}
\mathcal{G}=\bigcup_{k\in D } \mathcal{O}_k,
\end{align} where $D$ is some set, $\mathcal{O}_k$ are finite loops.
\end{proposition}
\begin{proof} 
A key property to be exploited is that, for any translating invariant event, its probability is either $0$ or $1$.
Here, the relevant translated invariant event is that, for any fixed edge $e_{a\rightarrow b}$ in any fixed graph $\mathcal{G}$, one can find finitely many edges so that they form a loop.
Assumption \ref{assu:uniqClus} implies that this will occur with probability 1.  

Consequently, almost surely, for any fixed edge $e_{a\rightarrow b}$, one can find finitely many edges $e_{b\rightarrow c_1}$, $e_{c_1\rightarrow c_2},$ $\cdots $ and $e_{c_k\rightarrow a}$ so that they form a finite loop $\mathcal{L}$ defined as $$\mathcal{L}:=\big\{ e_{a\rightarrow b},\ e_{b\rightarrow c_1},\ e_{c_1\rightarrow c_2},\ \cdots \ ,\ e_{c_k\rightarrow a}\big\}.$$

Delete this loop $\mathcal{L}$ from $\mathcal{G}$, and denote what is left by $\mathcal{G}\backslash \mathcal{L}$. By Assumption \ref{assu:uniqClus}, every edge in $\mathcal{G}\backslash \mathcal{L}$ is still contained in a finite loop, with probability 1.

By iterating this procedure, we prove the desired result: with probability 1, every graph is a union of finite loops. 
\end{proof}

Now, we apply Proposition \ref{prop:uniq} to obtain
the following useful observation.
\begin{lemma}\label{LM:GraWithSink}
Under Assumption \ref{assu:uniqClus}, as $L\rightarrow \infty$, almost surely, any unpaired graphs $\mathcal{G}_{\delta}(x\rightarrow y)$ and $\mathcal{G}_{0}(y\rightarrow x)$ generated by $\displaystyle\lim_{L\rightarrow \infty}\frac{\tilde{F}_{\delta,L}(x\rightarrow y)Z_{0,L}}{Z_{\delta,L}Z_{0,L}}$ and $\displaystyle\lim_{L\rightarrow \infty}\frac{F_{0,L}(y\rightarrow x)Z_{\delta,L}}{Z_{\delta,L}Z_{0,L}}$ can be decomposed into a union of finite loops $D_{\lambda}$ and a path $\mathcal{P}(x\rightarrow y)$ (or $\mathcal{Q}(y\rightarrow x)$ respectively) 
\begin{align}\label{eq:gDelta0}
\begin{split}
    \mathcal{G}_{\delta}(x\rightarrow y)=\Big(\bigcup_{\lambda\in \Upsilon_1} \mathcal{D}_{\lambda}\Big)\bigcup \mathcal{P}(x\rightarrow y);\\
    \mathcal{G}_{ 0}(y\rightarrow x)=\Big(\bigcup_{\lambda\in \Upsilon_2} \mathcal{D}_{\lambda}\Big)\bigcup \mathcal{Q}(y\rightarrow x).
\end{split}
\end{align}
\end{lemma}
\begin{proof}
We start with making Proposition \ref{prop:uniq} relevant. Proposition \ref{prop:uniq} works for graphs without a source or sink, while the desired result is for graphs with a source and a sink.

For that purpose, we fix a finite path $\mathcal{K}(y\rightarrow x)=\big\{ e_{y\rightarrow a_1}, \ e_{a_1\rightarrow a_2},\ \cdots,\ e_{a_l\rightarrow x} \big\}$, and consider the graph $\mathcal{G}_{\delta}(x\rightarrow y)\cup \mathcal{K}(y\rightarrow x),$ which has no sink or source. This makes Proposition \ref{prop:uniq} applicable, which says that, with probability 1, the path $\mathcal{K}(y\rightarrow x)$ is contained in a finite loop $\mathcal{L}$. 

We find a desired path $\mathcal{P}(x\rightarrow y)$ in \eqref{eq:gDelta0} by deleting $\mathcal{Q}(y\rightarrow x)$ from $\mathcal{L}$:
    \begin{align}
        \mathcal{P}(x\rightarrow y):= \mathcal{L}\backslash \mathcal{Q}(y\rightarrow x).
    \end{align}

    What is left is to consider the graph $\mathcal{G}_{\delta}(x\rightarrow y)\backslash \mathcal{P}(x\rightarrow y),$ which has no source or sink. Apply Proposition \ref{prop:uniq} to find that it is of the form, with probability 1,
    \begin{align*}
        \mathcal{G}_{\delta}(x\rightarrow y)\backslash \mathcal{P}(x\rightarrow y)=\Big(\bigcup_{\lambda\in \Upsilon_1} \mathcal{D}_{\lambda}\Big).
    \end{align*} This directly implies the first identity of the desired \eqref{eq:gDelta0}.

    The second identity can be obtained almost identically, thus we choose to skip the detail.
\end{proof}


\subsection{Edge pairing and weight assignments}\label{sub:renormalize}
We start with some preparatory works. 

In \eqref{eq:pairing}, where $L<\infty$, we assigned weights for each graph. As $L\rightarrow \infty$ 
we must assign weights for sets of infinite graphs.

For that purpose, we fix a large integer $N$ and define a notion called ``partially paired graph in the region $|z|\leq N$", denoted by $PP_N.$ Later, we will use it to represent a set of paired infinite graphs, see Definition \eqref{def:repres} below. Recall the definition of the projective map $\Pi$ from Definition \ref{def:projective}.
\begin{definition}
    For a fixed large integer $N$ so that $|x|+|y|<N$, a partially paired graph $G_{N}$ in the region $|z|\leq N$ satisfies the following properties:
    \begin{itemize}
    \item[(1)] Its underlying unpaired graph, $\Pi(G_N)$, is $\Big(\mathcal{G}_{\delta}(x\rightarrow y), \mathcal{G}_0\Big)\Big|_{|z|\leq N}$ or $\Big(\mathcal{G}_{\delta}, \mathcal{G}_0(y\rightarrow x)\Big)\Big|_{|z|\leq N}$. Recall unpaired graphs $\Big(\mathcal{G}_{\delta}(x\rightarrow y), \mathcal{G}_0\Big)$ or $\Big(\mathcal{G}_{\delta}, \mathcal{G}_0(y\rightarrow x)\Big)$ from \eqref{def:sdnn}.
    \item[(2)] In the interior, at site $z$ with $|z|<N$ and $z\not\in \{x,y\}$, every incoming edge is paired with an outgoing one; 
    \item[(3)] At the sites $x$ and $y$, the edges are not paired;
    \item[(4)] On the boundary, at site $z$ with $|z|=N$, the edges are not paired. 
    \end{itemize}
    We denote this set of graphs by $PP_{N}$, i.e.
    \begin{align}\label{def:partPair}
        PP_{N}:=\Big\{G_{N}\ \Big|\ G_{N}\ \text{is a partially paired graph in the region }\ |z|\leq N\Big\}.
    \end{align}
\end{definition}
We remark that the condition (4) seems strange. But it will be needed in \eqref{eq:count} to assign weights to sets of graphs.

In the next definition, we use a fixed partially-paired finite graph to represent a set of paired infinite graphs. Recall the definition of the projective map $\Pi$ in \eqref{def:projective}.
\begin{definition}\label{def:repres}
For any $$G_{\delta,N}(x\rightarrow y),\ G_{0,N}(y\rightarrow x)\in PP_{N},$$ we use them to represent two sets of paired infinite graphs, specifically,
\begin{align}\label{eq:PartPaired1}
  \begin{split}
      \widetilde{G_{\delta,N}}(x\rightarrow y) :=\Big\{ F\ \Big|& \ F=F_{\delta}(x\rightarrow y) \ \text{is a paired infinite graph;}\\ & F\Big|_{|z|< N}=G_{\delta,N}(x\rightarrow y)\Big|_{|z|<N};
       \ \Pi(F)\Big|_{|z|\leq N}=\Pi\big(G_{\delta,N}(x\rightarrow y)\big) \Big\},
  \end{split}
\end{align}
and
\begin{align}\label{eq:PartPaired2}
  \begin{split}
      \widetilde{G_{0,N}}(y\rightarrow x) :=\Big\{ F\ \Big|& \ F=F_{0}(y\rightarrow x) \ \text{is a paired infinite graph;}\\ & F\Big|_{|z|< N}=G_{0,N}(y\rightarrow x)\Big|_{|z|<N};
      \ \Pi(F)\Big|_{|z|\leq N}=\Pi\big(G_{0,N}(y\rightarrow x)\big) \Big\}.
  \end{split}
\end{align}
\end{definition}

Based on the discussions above, \eqref{eq:FdLZL1} and \eqref{eq:FdLZL2} take new forms: as $L\rightarrow \infty$,
\begin{align}\label{eq:WeightPre}
\begin{split}
    \lim_{L\rightarrow \infty}\frac{\tilde{F}_{\delta,L}(x\rightarrow y)\ Z_{0,L}}{\tilde{Z}_{\delta,L}Z_{0,L}}=\sum_{G_{\delta,N}(x\rightarrow y)\in PP_N} C_{\widetilde{G_{\delta, N}}(x\rightarrow y) } ,\\
    \lim_{L\rightarrow \infty}\frac{\tilde{Z}_{\delta,L}\ F_{0,L}(y\rightarrow x)}{\tilde{Z}_{\delta,L}Z_{0,L}}=\sum_{G_{0,N}(y\rightarrow x)\in PP_N } C_{\widetilde{G_{0,N}}(y\rightarrow x) }
\end{split}
\end{align} where the nonnegative constants $C_{\widetilde{G_{\delta,N}} (x\rightarrow y) }$ and $C_{\widetilde{G_{0,N}}( y\rightarrow x) }$ are weights assigned to the sets, and they are naturally defined.

We are ready to state two important properties for the weights. To simplify the notations we use the following conventions:
\begin{align}\label{eq:conventionDelta0}
G_{\#,N}=\left\{
\begin{array}{cc}
   G_{\delta,N}(x\rightarrow y)  & \text{if}\ \#=\delta, \\
   G_{0,N}(y\rightarrow x)  & \text{if}\ \#=0,
\end{array}
\right.
\end{align}
and $F_{\#,N},\ K_{\#,N}$ are interpreted similarly.
\begin{itemize}
\item Let $G_{\#,N}$ and $F_{\#,N}$ be two partially paired graphs in the region $|z|\leq N$ satisfying $$\Pi(G_{\#,N} )=\Pi(F_{\#,N} ),$$ then 
\begin{align}\label{eq:ctilde}
    C_{\widetilde{G_{\#,N}} }=C_{\widetilde{F_{\#,N}} }.
\end{align} This is implied by \eqref{eq:pairing} because for graphs with the same $\bf{n}({\bf{n}}_{\delta},{\bf{n}}_0)$, the weights are equal.

\item The weight $C_{\widetilde{G_{\#,N}}}$ is defined for any chosen large integer $N$. 
To show that these definitions are consistent for different $N_1$ and $N_2$ we use the following identity: suppose that $N_2\geq N_1$,
\begin{align}\label{eq:consistency}
    C_{\widetilde{G_{\#,N_1}}}=\sum_{
    \begin{array}{cc}
    G_{\#,N_2}\ \big|_{|z|< N_1}&=G_{\#,N_1};\\
    \Pi (G_{\#,N_2})\ \big|_{|z|\leq N_1}&=\Pi(G_{\#,N_1})
    \end{array}
    } C_{\widetilde{G_{\#,N_2}}},
\end{align} where the notations $$G_{\#,N_2}\ \big|_{|z|<N_1}=G_{\#,N_1}$$ and $$\Pi (G_{\#,N_2})\ \big|_{|z|\leq N_1}=\Pi(G_{\#,N_1})$$ signify that, (1) when restricted to the region $|z|<N_1$, $G_{\#,N_2}$ is the same to $G_{\#,N_1}$; (2) in the region $|z|\leq N_1$, $\Pi (G_{\#,N_2})$ is the same to $\Pi(G_{\#,N_1}).$

\end{itemize}


\subsection{Refined pairing and weight assignments}
Until now, we paired edges indiscriminately. An adverse fact is that some graphs in the set $\widetilde{G_{\delta,N}}(x\rightarrow y) $ don't have a paired finite $x\rightarrow y$-path at all. Thus, there is no possible switching to begin with!

To overcome this difficulty, we notice that Lemma \ref{LM:GraWithSink} implies that, with probability 1, any (unpaired) graph has a finite $x\rightarrow y$ path. Based on this, to ensure there is always a ``switchable" paired $x\rightarrow y$-path, we refine the paring by introducing the notion of ``finitely-paired graphs".
\begin{definition}
For a fixed large integer $N$, a paired infinite graph $G(x\rightarrow y)$ is called a finitely-paired graph in the region $|z|\leq N$ if every edge in the region $|z|\leq N$ is paired into a finite loop or a finite path.
\end{definition}

To illustrate the ideas we use the following  example of finitely-paired graph in the region $|z|\leq N$, see Figure \ref{fig:my_image}.
  \begin{figure}[htbp]
           \centering
           \includegraphics[width=0.5\textwidth]{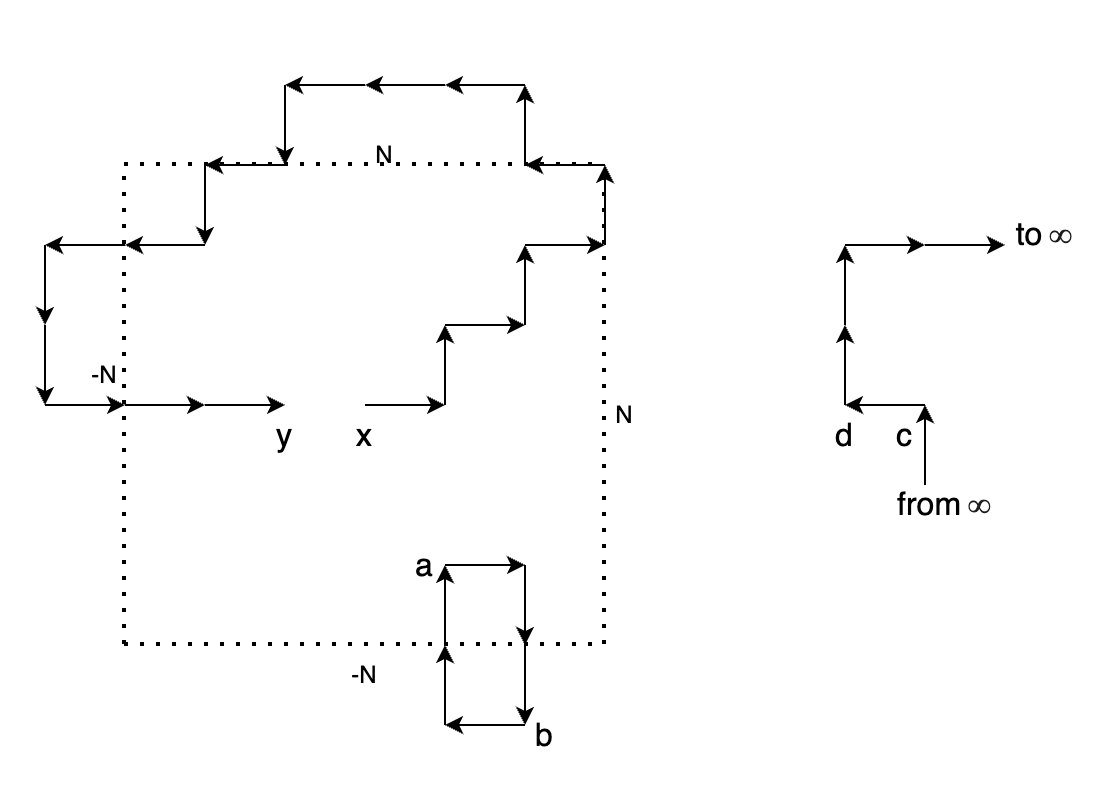}
           \caption{An example for finitely-paired graph}
           \label{fig:my_image}
    \end{figure}
In the figure part of the paired path $x\rightarrow y$ and paired loop $a\rightarrow b\rightarrow a$ are in the $|z|\leq N$ region, and the path and loop are of finite length. If the path $\infty\rightarrow c\rightarrow d\rightarrow \infty$ does not intersect with the region $|z|\leq N$, then we do not require that it is of finite length.

To simplify the notations we use the same conventions as in \eqref{eq:conventionDelta0}.
For a fixed partially paired finite graph $$G_{\#,N}=G_{\delta,N}(x\rightarrow y)\ \text{ or}\ G_{0,N}(y\rightarrow x),$$ we use it to represent a set of paired infinite graphs, denoted by $\widehat{G_{\#,N}}$:
\begin{align}
\begin{split}
    \widehat{G_{\#,N}}:=\Big\{ F\ \Big|&\ F=F_{\#}\ \text{is a paired infinite graph in } \ \mathbb{Z}^3; \ F\Big|_{< N}=G_{\#,N};\\ & \Pi(F)|_{\leq N}=\Pi(G_{\#,N});\ \ F\ \text{ is finitely-paired in the region} \ |z|\leq N.\Big\}.
\end{split}
\end{align}

We rewrite \eqref{eq:WeightPre} to assign weight for $\widehat{G_{\#,N}}$: for a fixed large $N,$
\begin{align}\label{eq:finitelyLooped}
\begin{split}
     \lim_{L\rightarrow \infty}\frac{\tilde{F}_{\delta,L}(x\rightarrow y)\ Z_{0,L}}{\tilde{Z}_{\delta,L}Z_{0,L}}=&\sum_{G_{\delta,N}(x\rightarrow y)\in PP_N } C_{\widetilde{G_{\delta,N}}(x\rightarrow y) }=\sum_{G_{\delta,N}(x\rightarrow y) \in PP_N} D_{ \widehat{G_{\delta,N}}(x\rightarrow y) },\\
     \lim_{L\rightarrow \infty}\frac{F_{0,L}(y\rightarrow x)\ \tilde{Z}_{\delta,L}}{\tilde{Z}_{\delta,L}Z_{0,L}}=&\sum_{G_{0,N}(y\rightarrow x)\in PP_N } C_{\widetilde{G_{0,N}}(y\rightarrow x) }=\sum_{G_{0,N}(y\rightarrow x)\in PP_N } D_{ \widehat{G_{0,N}}(y\rightarrow x) }
\end{split}
\end{align} where  $D_{\widehat{G_{\#,N}}}$
are weights assigned to $\widehat{G_{\#,N}}$, and are uniquely defined by the following two identities: 
\begin{itemize}
    \item[(B1)] 
$D_{\widehat{G_{\#,N}}}$ depends on $C_{\widetilde{G_{\#,N}}}$: for a fixed unpaired graph $\mathcal{K}_{\#,N}$ in the region $|z|\leq N,$
    \begin{align}\label{eq:uniform}
   \sum_{\Pi(G_{\#,N})=\mathcal{K}_{\#,N}} D_{\widehat{G_{\#,N}}}:=  \sum_{\Pi(G_{\#,N})=\mathcal{K}_{\#,N}} C_{\widetilde{G_{\#,N}}}.
\end{align} 

    \item[(B2)] Suppose that $G_{\#,N}$ and $F_{\#,N}$ are two partially paired graphs in the region $|z|\leq N$, and their underlying unpaired graphs agree, i.e. $$\Pi(G_{\#,N})=\Pi(F_{\#,N}),$$ then they are of the same weights:
    \begin{align}\label{eq:count}
        D_{\widehat{G_{\#,N} }}=D_{\widehat{F_{\#,N} }}.
    \end{align} 
\end{itemize}

The identities \eqref{eq:finitelyLooped}, \eqref{eq:uniform} and \eqref{eq:count} assign a unique finite weight to $D_{ \widehat{G_{\#,N}}}$ because, for each fixed $N$ and an unpaired graph $\mathcal{K}_{\#,N}$ in \eqref{eq:uniform}, there are only finitely many $\widehat{G_{\#,N}}$ satisfying $\Pi(G_{\#,N})=\mathcal{K}_{\#,N}$ .

We also need a consistency, similar to \eqref{eq:consistency},
\begin{lemma}
    The identities in \eqref{eq:finitelyLooped} are consistent for different $N$'s in the sense that for any $M>N$, 
    \begin{align}\label{eq:consist}
        D_{\widehat{G_{\#,N}}(x\rightarrow y) }=\sum_{
\begin{array}{cc}
        &W_{\#,M}\big|_{|z|<N}=G_{\#,N} \\ 
        &\Pi(W_{\#,M})\big|_{|z|\leq N}=\Pi(G_{\#,N}) 
\end{array}
        } D_{\widehat{W_{\#,M}}(x\rightarrow y) }.
    \end{align}
\end{lemma}

We continue to study the paired graphs after refining the pairings.

Refining the pairing is an improvement because it guarantees there is a graph for switching, and the switching is a bijection. 

However, we have not proved that it is between graphs of equal weights.
For this purpose, we need the following result. 
\begin{proposition}\label{prop:keyWeights}
Suppose that $P=P(x\rightarrow y)$ is a union of paired paths $x\rightarrow y$, $y\rightarrow x$, and loops $x\rightarrow x$, $y\rightarrow y$ in the region $|z|\leq N$, except that at $x$ and $y$ the edges are not paired.

Suppose that $G_{\delta,N}(x\rightarrow y)$ is a partially paired graph in the region $|z|\leq N$, and $P\subset G_{\delta,N}(x\rightarrow y) $ contains all the paired paths and loops connecting $x$ and $y$.
$G_{0,N}(y\rightarrow x)$ is the partially paired graph obtained by switching $P$ in $G_{\delta,N}(x\rightarrow y)$

Then $\widehat{G_{\delta,N}}(x\rightarrow y)$ and $\widehat{G_{0,N}}(y\rightarrow x)$ have the same weights
    \begin{align}\label{eq:equalWeights}
        D_{\widehat{G_{\delta,N}}(x\rightarrow y) }=D_{\widehat{G_{0,N}}(y\rightarrow x) }.
    \end{align}
\end{proposition}
This proposition will be proved in Subsection \ref{subsec:KeyWeights}.

Proposition \ref{prop:keyWeights} is important because it implies the important Proposition \ref{Prop:limitDxy}.


\subsubsection{Proof of Proposition \ref{Prop:limitDxy} by assuming Proposition \ref{prop:keyWeights}}
We start with an easy case. 

Specifically, suppose that the partially paired graphs $G_{\delta,N}(x\rightarrow y)$ and $G_{0,N}(y\rightarrow x)$ in the region $|z|\leq N$ are related as follows: A paired graph $P(x\rightarrow y)$ is subgraph of $ G_{\delta,N}(x\rightarrow y)$, i.e.
\begin{align}\label{eq:condi1}
P(x\rightarrow y)\subset G_{\delta,N}(x\rightarrow y)
\end{align}
and it contains all the paths and loops connected to $x$ and $y$, and one obtains $G_{0,N}(y\rightarrow x)$ by switching $P(x\rightarrow y),$ i.e.
\begin{align}\label{eq:condi2}
G_{\delta,N}(x\rightarrow y)\xrightarrow{P} G_{0,N}(y\rightarrow x).
\end{align}
Then we apply Proposition \ref{prop:keyWeights} to obtain,
\begin{align}
     D_{\widehat{G_{\delta,N}}(x\rightarrow y) }=D_{\widehat{G_{0,N}}(y\rightarrow x) }.\label{eq:sameWei}
\end{align}

This is important because, by \eqref{eq:finitelyLooped}, 
\begin{align*}
&\displaystyle\lim_{L\rightarrow \infty}\Big(\frac{\tilde{F}_{\delta,L}(x\rightarrow y)\ Z_{0,L}}{\tilde{Z}_{\delta,L}Z_{0,L}}-\frac{F_{0,L}(y\rightarrow x)\ \tilde{Z}_{\delta,L}}{\tilde{Z}_{\delta,L}Z_{0,L}}\Big)\\
=&\sum_{G_{\delta,N}(x\rightarrow y)\in PP_{N} } D_{ \widehat{G_{\delta,N}}(x\rightarrow y) }-\sum_{G_{0,N}(x\rightarrow y)\in PP_N } D_{ \widehat{G_{0,N}}(y\rightarrow x) }.
\end{align*} Recall the definition of partially paired graphs, $PP_N$, in Definition \eqref{def:partPair}. 
Then \eqref{eq:sameWei} implies that some terms cancel each other under the condition \eqref{eq:condi1}.

We have a minor difficulty in proving that all the terms $D_{ \widehat{G_{\delta,N}}(x\rightarrow y) }$ will be canceled because \eqref{eq:condi1} might not hold. 
For this, we need the consistency between different $N$s in \eqref{eq:consist} and need a standard argument, which says that if $N$ is sufficiently large, then up to a correction of order $\epsilon_{N}$, with $$\lim_{N\rightarrow \infty}\epsilon_{N}=0,$$ all the paths and loops connecting $x$ and $y$ are in the region $|z|\leq N.$ The detailed proof is straightforward but tedious. We choose to skip it.

Similarly, almost all terms $D_{ \widehat{G_{0,N}}(y\rightarrow x) }$ will be canceled.

Therefore, we proved that
\begin{align}
   \lim_{L\rightarrow \infty}\Big(\frac{\tilde{F}_{\delta,L}(x\rightarrow y)\ Z_{0,L}}{\tilde{Z}_{\delta,L}Z_{0,L}}-\frac{F_{0,L}(y\rightarrow x)\ \tilde{Z}_{\delta,L}}{\tilde{Z}_{\delta,L}Z_{0,L}}\Big)= 0.
\end{align}

This and the definition of $ D_{1,x,y}(L)$ in \eqref{eq:differe} imply the desired $\displaystyle\lim_{L\rightarrow \infty} D_{1,x,y}(L)=0$. 

Similarly, we prove that $\displaystyle\lim_{L\rightarrow \infty}D_{2,x,y}(L)=0$, and hence finish proving Proposition \ref{Prop:limitDxy}.


\subsubsection{Switching between graphs of equal weights: Proof of Proposition \ref{prop:keyWeights}}\label{subsec:KeyWeights}
For technical reasons, it is easier to prove a slightly more general result, which will be Lemma \ref{LM:ReSurg} below.

We start by defining a notion of ``surgical switching by a graph $\mathcal{P}$" for some unpaired graph $\mathcal{P}=\mathcal{P}(x\rightarrow y)$.  
Here we emphasize that $\mathcal{P}$ is unpaired, and thus, in switching we do not follow the rule of pairing, and hence, we destroy and re-establish some of them. 

Recall that $\Pi$, defined in \eqref{def:projective}, is the projective map so that $\Pi(G_{\#}(x\rightarrow y))$ is the unpaired graph underlying the paired one $G_{\#}(x\rightarrow y)$. 

\begin{definition}
Suppose that $G_{\delta,\mathcal{P}}=G_{\delta,\mathcal{P}}(x\rightarrow y)$ is a paired infinite graph. And $\mathcal{P}=\mathcal{P}(x\rightarrow y)$ is an unpaired subgraph in $\Pi(G_{\delta,\mathcal{P}})$.

We perform a surgical switching, using $\mathcal{P}$, on $G_{\delta,\mathcal{P}}$ to obtain $G_{0,\overline{\mathcal{P}}}=G_{0,\overline{\mathcal{P}}}(y\rightarrow x)$ as follows.

For $G_{0,\overline{\mathcal{P}}}$, its corresponding unpaired graph $\Pi(G_{0,\overline{\mathcal{P}}})$, is obtained from $\Pi(G_{\delta})$ by switching $\mathcal{P}$ with the procedure described in \eqref{eq:switch}.
 
Next, we consider the pairing in the new graph $G_{0,\overline{\mathcal{P}}}$. 

In brief, some pairings remain unchanged, and some pairings must be modified depending on whether they are paired with edges in $\mathcal{P}$. 

\begin{itemize}
\item[(\text{I})] We start with an easy case. Specifically, suppose that $\mathcal{P}$ is a $x\rightarrow y$ path that does not contain any loop and is of the form 
\begin{align}\label{eq:noLoop}
\begin{split}
x&=a_{0}\rightarrow a_{1}\rightarrow a_{2}\rightarrow \cdots \rightarrow a_{n}=y,\\
&\text{and}\ a_{k}\not=a_{l}\ \text{if}\ k\not=l.
\end{split}
\end{align}

Fix a site $v\in \mathbb{Z}^3$, the pairing of an edge $e_{v\rightarrow w}$ (or $e_{w\rightarrow v}$) at the site $v$ will stay the same in the following two possibilities:
\begin{itemize}
\item the edge is not paired with any edge of $\mathcal{P}$,
\item the edge $e_{v\rightarrow w}$ (or $e_{w\rightarrow v}$) is part of $\mathcal{P}$ and is already paired to another edge in $\mathcal{P}$ at $v$. 
\end{itemize}
    
Otherwise, $v=a_k$ for some $k=1,2,\cdots,n-1$ there are two possible pairings of $e_{a_k\rightarrow w}$ (or $e_{w\rightarrow a_k}$) at the site $a_k$, and we handle them differently. In what follows we only consider $e_{a_k\rightarrow w}$ at the site $a_k$. That for $e_{w\rightarrow a_k}$ is almost identical, and hence is omitted.
\begin{itemize}
\item  The edge $e_{a_{k}\rightarrow w}$ is part of $\mathcal{P}$, thus $w=a_{k+1}$ by \eqref{eq:noLoop}. At $a_k$ it is paired with an edge $e_{b_{k-1}\rightarrow a_k}\not\in \mathcal{P}$. And the edge $e_{a_{k-1}\rightarrow a_k}\in \mathcal{P}$ is paired with $ e_{a_k\rightarrow c_{k+1}}\not\in \mathcal{P}$. 

Then in the new graph $G_{0,\overline{\mathcal{P}}}$ where we reverse the path $\mathcal{P}$, at $a_k$, $e_{a_{k}\leftarrow a_{k+1}}$ is paired with $e_{a_k\rightarrow c_{k+1}}$; and $e_{a_{k-1}\leftarrow a_k}$ is paired with $e_{a_k\leftarrow b_{k-1}}$. 

\item The edge $e_{a_{k}\rightarrow w}$ is not part of $\mathcal{P}$, and at $a_k$ it is paired with $e_{a_{k-1}\rightarrow a_k}\in \mathcal{P},$ and the edge $e_{a_k\rightarrow a_{k+1}}\in \mathcal{P}$ is paired with an edge $e_{c\rightarrow a_k}$. 

Then in the new graph $G_{0,\overline{\mathcal{P}}}$ where we reverse the path $\mathcal{P}$, at $a_k,$ $e_{a_{k}\rightarrow w}$ is paired with $e_{a_k\leftarrow a_{k+1}}$, and $e_{c\rightarrow a_k}$ is paired with $e_{a_{k-1}\leftarrow a_k}$.

\end{itemize}
To illustrate the procedure, we use the Figures \ref{image2} and \ref{image3}. Both are paired graphs. In Figure \ref{image2}, the paired path and loop are $x\rightarrow a1\rightarrow a2\rightarrow y$ and $b1\rightarrow a1\rightarrow b2\rightarrow b3\rightarrow a2\rightarrow b4\rightarrow b1$. And the chosen (unpaired) path for surgical switching, $\mathcal{P}$, is $$\mathcal{P}=x\rightarrow a1\rightarrow b2\rightarrow b3\rightarrow a2\rightarrow y.$$ Figure \ref{image3} is obtained after the surgical switching, the paired path and loop are $y\rightarrow a2\rightarrow b4\rightarrow b1\rightarrow a1\rightarrow x$ and $a1\rightarrow a2\rightarrow b3\rightarrow b2\rightarrow a1.$
    
\item[(\text{II})] For the general case, $\mathcal{P}$ can be decomposed, disjointly, into 
    \begin{align}\label{eq:decomP}
    \mathcal{P}=\mathcal{P}_1\cup \mathcal{P}_2\cup \cdots \mathcal{P}_{m}
    \end{align}
    and each of $\mathcal{P}_k$ is a self-avoiding path or loop, then we make the surgical switching for each one of them similarly.
    \end{itemize}
\end{definition}

\begin{figure}[htbp]
           \centering
           \includegraphics[width=0.5\textwidth]{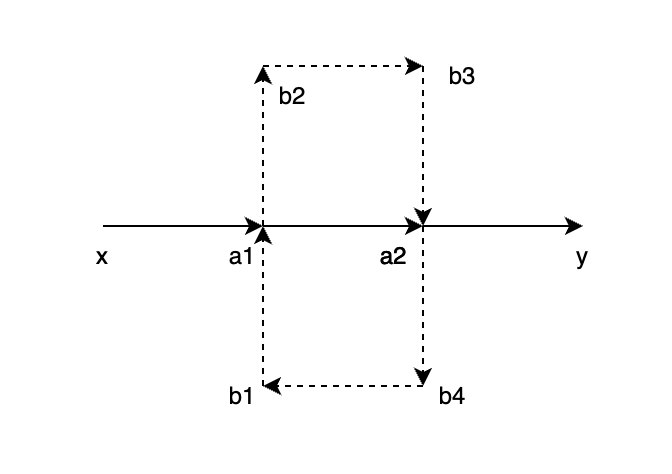}
           \caption{Before surgical switching}
           \label{image2}
    \end{figure}

\begin{figure}[htbp]
           \centering
           \includegraphics[width=0.5\textwidth]{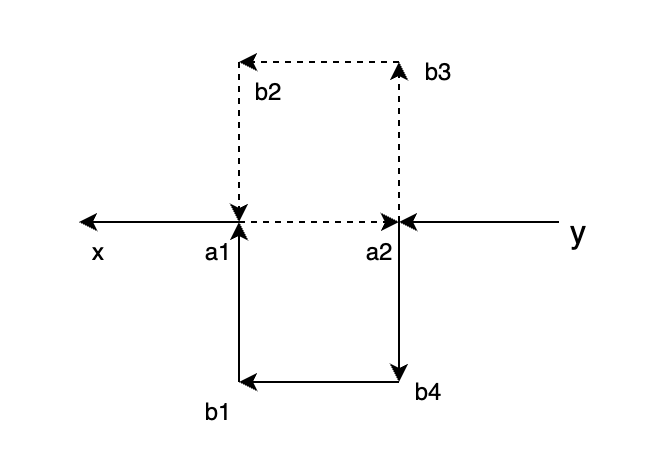}
           \caption{After surgical switching}
           \label{image3}
\end{figure}

To facilitate later discussions we use the following notation for the surgical switching discussed above:
\begin{align}\label{def:surgicalTran}
    G_{\delta,\mathcal{P}}\xrightarrow[\text{surgical}]{\mathcal{P}}G_{0,\overline{\mathcal{P}}}.
\end{align}

A keen reader might notice a problem: Because the decomposition of $\mathcal{P}$ in \eqref{eq:decomP} might not be unique, $ G_{\delta,\mathcal{P}}\xrightarrow[\text{surgical}]{\mathcal{P}}G_{0,\overline{\mathcal{P}}}$ can be defined differently. 
This is true. However, this will not cause a problem because, instead of considering the mapping between two paired graphs, we study the mapping between sets of paired graphs where all the possible pairings with edges of $\mathcal{P}$ are included, and hence the decomposition in \eqref{eq:decomP} is not important.

The next result is obvious.
\begin{lemma}
    For any fixed finite graph $\mathcal{P}=\mathcal{P}(x\rightarrow y)$, the mapping $\xrightarrow[\text{surgical}]{\mathcal{P}}$ is one-to-one.
\end{lemma}


We are ready to state one more important property, besides the two in \eqref{eq:ctilde} and \eqref{eq:consistency}, for the weights $C_{\widetilde{G_{\#,N}}}$, defined in \eqref{eq:WeightPre}, with $\#=\delta$ or $0.$ In what follows we use the conventions in \eqref{eq:conventionDelta0}.
\begin{itemize}
\item 
If $G_{\delta,N}  \xrightarrow[surgical]{\mathcal{P}} K_{0,N} $ for some graph $\mathcal{P}=\mathcal{P}(x\rightarrow y)$ contained in the region $|z|<N$, then
\begin{align}\label{eq:surgicalLoop1}
    C_{\widetilde{G_{\delta,N}} }=C_{\widetilde{K_{0,N}} }.
\end{align}

Similar to \eqref{eq:ctilde}, it is implied by \eqref{eq:pairing} with a similar justification, and recall that we do not pair edges at the sites $x$ and $y.$

\end{itemize}


By now we have finished analyzing the weights $C_{\widetilde{G_{\#,N}}}$, we continue to study the weights of the graph with refined pairing $D_{\widehat{G_{\#,N}}},$ defined in \eqref{eq:finitelyLooped}.

The next property implies the desired Proposition \ref{prop:keyWeights}.
\begin{lemma}\label{LM:ReSurg}
     If $G_{\delta,N}(x\rightarrow y)$ and $F_{0,N}(y\rightarrow x)$ are two partially paired graphs in the region $|z|\leq N$, and $\mathcal{P}=\mathcal{P}(x\rightarrow y)\subset \Pi(G_{\delta,N}(x\rightarrow y))$ is a unpaired graph contained in the region $|z|<N$, such that $$G_{\delta,N}(x\rightarrow y)\xrightarrow[\text{surgical}]{\mathcal{P}} F_{0,N}(y\rightarrow x),$$ then
\begin{align}\label{eq:ReSurg}
    D_{\widehat{G_{\delta,N}}(x\rightarrow y)}=D_{\widehat{F_{0,N}}(y\rightarrow x)}.
\end{align}
\end{lemma}
\begin{proof}
To simplify notations we use the same conventions as in \eqref{eq:conventionDelta0}.

We start from \eqref{eq:surgicalLoop1} and \eqref{eq:uniform}, which imply that 
\begin{align}\label{eq:iden1}
\sum_{\Pi(V_{\delta,N})=\Pi(G_{\delta,N})} D_{\widehat{V_{\delta,N}}  }
    =\sum_{\Pi(K_{0,N})=\Pi(F_{0,N})}D_{\widehat{K_{0,N}} }.
\end{align}

Our strategy is to show that, (1) on both sides of \eqref{eq:iden1}, the summations are taken over terms of equal weights; (2) the number of elements in the two summations are equal. Thus, every term involved is of equal weight as desired.

For the left hand side, the rule set in \eqref{eq:count} implies that $D_{\widehat{V_{\delta,N}} }=D_{\widehat{G_{\delta,N}} }$ in \eqref{eq:count}. Thus, 
\begin{align}\label{eq:iden2}
    \begin{split}
    &\sum_{\Pi(V_{\delta,N})=\Pi(G_{\delta,N})} D_{\widehat{V_{\delta,N}} }
    =\Upsilon_{G_{\delta,N} }\ 
    D_{\widehat{G_{\delta,N}} },
    \end{split}
\end{align} where $\Upsilon_{G_{\delta,N} }$ is used to count the number of elements in the summation:
\begin{align*}
    \Upsilon_{G_{\delta,N} }:=&\Big|\Big\{ \widehat{V_{\delta,N}}  \ \Big| \  \Pi(V_{\delta,N} )=\Pi(G_{\delta,N} )  \Big\} \Big|.
\end{align*}

Similarly, the right-hand side takes a new form
\begin{align}\label{eq:iden3}
\begin{split}
 \sum_{\Pi(K_{0,N})=\Pi(F_{0,N})}D_{\widehat{K_{0,N}} } 
 = \Upsilon_{F_{0,N} }  D_{\widehat{F_{0,N}} } 
\end{split}
\end{align} with 
\begin{align*}
    \Upsilon_{F_{0,N} }:=&\Big|\Big\{ \widehat{K_{0,N}} \ \Big| \  \Pi(K_{0,N} )=\Pi(F_{0,N} )  \Big\} \Big|.
\end{align*}

The three identities \eqref{eq:iden1}-\eqref{eq:iden3} imply that
\begin{align}\label{eq:equalNo}
\begin{split}
     \Upsilon_{G_{\delta,N} } D_{\widehat{G_{\delta,N}} }
   = \Upsilon_{F_{0,N} } D_{\widehat{F_{0,N}}  }.
\end{split}    
\end{align}  

Now we claim that
\begin{align}\label{eq:upsilon}
\Upsilon_{G_{\delta,N} }=\Upsilon_{F_{0,N} }.
\end{align} If this holds, then it implies the desired identity $$D_{\widehat{G_{\delta,N}} }=D_{\widehat{F_{0,N}}  }.$$ 

To prove \eqref{eq:upsilon}, we observe that for the $\mathcal{P}$ provided in the lemma, the surgical transformation $\xrightarrow[\text{surgical}]{\mathcal{P}}$, defined in \eqref{def:surgicalTran}, is a bijection between the sets $\Big\{ \widehat{V_{\delta,N}}  \ \Big| \  \Pi(V_{\delta,N} )=\Pi(G_{\delta,N} )  \Big\}$ and $\Big\{ \widehat{K_{0,N}}  \ \Big| \  \Pi(K_{0,N}  )=\Pi(F_{0,N} )  \Big\}$. Therefore, these two finite sets have the same number of elements, and hence \eqref{eq:upsilon} holds.

The proof is complete.  
\end{proof}

\end{document}